\documentclass[11pt,letterpaper]{amsart}%
\usepackage{amsmath}
\usepackage{txfonts}
\usepackage{graphicx}
\usepackage{amsthm}
\usepackage{amsfonts}
\usepackage{amssymb}
\usepackage[utf8]{inputenc}
\usepackage{comment}
\usepackage{xcolor}
\usepackage{lipsum}
\usepackage{ifthen}%
\usepackage{tikz}

\usepackage{newtx}
\usepackage{relsize}

\providecommand{\U}[1]{\protect\rule{.1in}{.1in}}
\numberwithin{equation}{section}
\newtheorem{theorem}{Theorem}[section]
\newtheorem{lem}[theorem]{Lemma}
\newtheorem{thm}[theorem]{Theorem}
\newtheorem{pro}[theorem]{Proposition}
\newtheorem{cor}[theorem]{Corollary}
\newtheorem{defi}[theorem]{Definition}
\newtheorem{rem}[theorem]{Remark}

\def\Area{\operatorname{Area}\,}
\def\Re{\operatorname{Re}}
\def\Res{\operatorname{Res}}
\def\euc{\operatorname{euc}}
\def\loc{\operatorname{loc}}

\def\S2{\mathbb{S}^2}
\def\s{\,\,\,\,}

\def\R{\mathbb{R}}

\def\S{\mathcal{S}}
\def\M{\mathcal{M}}
\def\K{\mathbb{K}}

\allowdisplaybreaks
\def\s{\,\,\,\,}

\def\R{\mathbb{R}}

\def\S{\mathcal{S}}
\def\M{\mathcal{M}}
\def\K{\mathbb{K}}
\def\res{\mathrm{Res}}

\allowdisplaybreaks

\title[Prescribing positive curvature with singularities]{Prescribing positive curvature with conical singularities on $\mathbb S^2$}

\author{Jingyi Chen, Yuxiang Li, Yunqing Wu }
\address{Department of Mathematics\\
	The University of British Columbia, Vancouver, Canada}
\email{jychen@math.ubc.ca}
\address{
	Department of Mathematics\\
	Tsinghua University, Beijing, China}
\email{liyuxiang@mail.tsinghua.edu.cn}
\address{
	The Institute of Geometry and Physics\\
	University of Science and Technology of China, Hefei, Anhui, People's Republic of China}
\email{yqwu19@ustc.edu.cn}

\thanks{Chen is partially supported by NSERC Discovery Grant GR010074}
\thanks{Li is partially supported by  National Key R\&D Program of China 2022YFA1005400}

\allowdisplaybreaks[4]

\date{}

\begin{document}
	
	\maketitle

	\vspace{-.2in}
	
	\begin{abstract}
		For conformal metrics with conical singularities and positive curvature on $\mathbb S^2$, we prove a convergence theorem and apply it to obtain a criterion for nonexistence in an open region of the prescribing data. The core of our study is a fine analysis of the bubble trees and an area identity in the convergence process. 
	\end{abstract}

	\tableofcontents
	
	\section{Introduction}
	
	A real divisor $\mathfrak{D}$ on a compact surface $\Sigma$ is a formal sum $
	\mathfrak{D}=\sum_{i=1}^{m} \beta_{i} p_{i}, $
	where $\beta:=(\beta_1,\cdots,\beta_m)\in\R^m$ and $p_{i} \in \Sigma$ are distinct. The Euler characteristic of the pair $(\Sigma,\beta)$ is defined to be
	\begin{align*}
	\chi(\Sigma,\beta):=\chi(\Sigma)+\sum_{i=1}^{m} \beta_{i}. 
	\end{align*}
	Let $g_0$ be a Riemannian metric on $\Sigma$. A conformal metric $g$ on $(\Sigma,g_{0})$ is said to represent the divisor $\mathfrak{D}$ if $g$ is a smooth metric away from $p_{1},\cdots,p_{m}$ such that around each $p_i$ there is an isothermal coordinate neighbourhood $U_{i}$ w.r.t. $g_0$ with a coordinate $z_{i}$ such that $z_{i}(p_{i})=0$ and $g$ is in the form
	\begin{align*}
	g=e^{2v} |z_{i}|^{2 \beta_{i}}g_{\euc},
	\end{align*}
	where $v \in C^0(U_{i}) \cap C^{2}(U_{i} \backslash \{p_{i}\})$. 
	The point $p_{i}$ is called a conical singularity of angle $\theta_i=2\pi (\beta_{i}+1)$ if $\beta_{i}>-1$ and a cusp if $\beta_i=-1$. 
	
	Finding a conformal metric $g=e^{2u}g_{0}$ representing $\mathfrak{D}$ with a prescribed function $K$ on $\mathbb \Sigma$ is equivalent to solving
	\begin{align}
	\label{singular conformal metric equation}
	-\Delta_{g_{0}} u= K e^{2u} -K_{g_{0}}-2 \pi \sum_{i=1}^{m} \beta_{i} \delta_{p_{i}}
	\end{align}
	in the sense of distribution, where $\delta_{p_{i}}$ is the Dirac measure at $p_{i}$.
	
	When $\chi(\Sigma,\beta)>0$ and $\mathfrak D$ consists of conical singularities, existence of a conformal metric representing $\mathfrak D$ is known in a few cases.
	For instance, it is shown that \eqref{singular conformal metric equation} admits a solution if $0<\chi(\Sigma,\beta)<2\min_i\{1,\beta_i+1\}$ (which is called the Trudinger constant), $\beta_1,...,\beta_m>-1$ and $K$ is positive somewhere \cite[Theorem C]{Troyanov1991} and 
	if $\chi(\Sigma,\beta)>2$, the genus of $\Sigma$ is at least 1, $\beta_1,...,\beta_m>0$ and $\beta$ not in certain lattice, $K$ is a positive  Lipschitz continuous function \cite[Theorem 1.1]{bartolucci2011supercritical}. In \cite{luo1992liouville}, F. Luo and G. Tian established the uniqueness of Troyanov's solution to Liouville's equation on the punctured complex plane. For Liouville type equations, uniqueness and symmetry of solutions are studied in \cite{GM, BCJM, SSTW}, etc. and existence via a topological degree theory in \cite{li1999harnack, Lin, ChenLin03, ChenLin15}; we refer the reader to the reference therein for many other works. 
	
	For the spherical metric of constant curvature on $\mathbb S^2$, important classifications of conformal metrics with conical singularities has been achieved. It is shown that a solution cannot have exactly one singularity by W.X. Chen and C.M. Li \cite{CW-L}, a solution with exactly two conical singularities must satisfy a strong rigidity by M. Troyanov \cite{Troyanov1988}, and A. Eremenko \cite{Eremenko2004} treated three conical singularities. We will use the first two cases in our proof of results stated below. 
	
	When the singular points $\{p_1,...,p_m\}$ are not fixed on $\mathbb S^2$, the question of finding a spherical metric with prescribed angles at $m$ conical singularities has been studied, in particular, under a constrain on the holonomy representation: $ \pi_1(\mathbb S^2\backslash\{p_1,...,p_m\}) \to SO(3)$, in \cite{Mondello-P, Dey2018, Eremenko2017, kapovich}. Inspired by these works, 
	although the points in $\mathfrak{D}$ are fixed in our consideration, we introduce a function to describe the type of divisors in our interest: 
	
	First, let $\mathbb F: \mathbb R^m \to \mathscr{P}(\{1,\cdots,m\})$ be an index function for certain half spaces of $\mathbb R^m$ where $\mathscr P$ stands for the power set, given by 
	$$
	\mathbb{F}(  (v_{1}, \cdots,v_{m})  )=\Big\{i:v_i\leq\frac{1}{2}\sum^m_{k=1}v_k\Big\}.
	$$
	Then, we define 
	\begin{align*}
	&\mathcal{A}_m=\Big\{\beta=(\beta_{1},\cdots,\beta_{m})  \in\mathbb R^m:-1<\beta_1,\cdots,\beta_m\neq 0,\, \frac{1}{2}\chi(\mathbb S^2,\beta) \in (0,+\infty) \backslash \mathbb{N},\\ 
	&\text{ and for } \forall  J\subset\mathbb{F}(\beta), \,\frac{1}{2}\chi(\mathbb S^2,\beta)-\sum_{j\in J}\beta_{j}\neq |J|, |J|+1, \cdots, |J|+\big[ \frac{1}{2} \chi( \mathbb{S}^{2} ) \big]\Big\}.
	\end{align*}

	There are divisors on $\mathbb S^2$ with $\beta$ in $\mathcal{A}_m$ that are not representable by the spherical metric (cf. \cite{Mondello-P, Dey2018}); in other words, \eqref{singular conformal metric equation} with such  $\beta$ and $K=K_{g_0}=1$ admits no solutions. The {\it nonexistence} persists in an open neighbourhood of $(\beta,1)$ as a special case of our results in this paper.

	\begin{theorem}
		\label{thm openness of nonexistence K C0}
		Let $K$ lie in the set $C^{+}(\mathbb{S}^{2})$ of positive continuous functions on $\mathbb S^2$. Suppose that $\beta\in\mathcal{A}_m$ and one of the following assumptions holds:
		\begin{enumerate}
			\item[($\mathscr{O}_1$)]  $0<\chi(\mathbb{S}^{2},\beta)<2$; 
			\item[($\mathscr{O}_2$)]  $\beta_{i} \leq 1$ for any $i \in \{1,\cdots,m\}$. 
		\end{enumerate}
		If the equation 
		\begin{align*}
		-\Delta_{\mathbb{S}^{2}} u=Ke^{2u}-2\pi\sum_{i=1}^{m} \beta_i\delta_{p_i}-1
		\end{align*}
		has no solutions in $\cap_{p\in [1,2)} W^{1,p}(\mathbb{S}^{2},g_{\mathbb{S}^{2}})$, then there exists a neighbourhood $\mathscr{U}$ of  $(K,\beta)$ in $C^{+}(\mathbb{S}^{2}) \times (-1,\infty)^{\times m}$ such that for any $( \tilde{K},\tilde{\beta} ) \in \mathscr{U}$, 
		\begin{align*}
		-\Delta_{\mathbb{S}^{2}} u=\tilde{K} e^{2u}-2\pi\sum_{i=1}^{m} \tilde{\beta}_i\delta_{p_i}-1
		\end{align*}
		has no solutions in $\cap_{p\in [1,2)} W^{1,p}(\mathbb{S}^{2}, g_{\mathbb{S}^{2}})$. 
	\end{theorem}

	In fact,  stronger regularity on $K$ allows us to  drop the assumption ($\mathscr{O}_1$) and ($\mathscr{O}_2$) in Theorem \ref{thm openness of nonexistence K C0}: 
	
	\begin{theorem}
		\label{thm openness of nonexistence K C1}
		Let $K$ lie in the set $C^{1,+}(\mathbb{S}^{2})$ of positive $C^1$-smooth functions on $\mathbb S^2$. Suppose  $\beta =(\beta_1,\cdots,\beta_m) \in \mathcal{A}_m$.
		If the equation 
		\begin{align*}
		-\Delta_{\mathbb{S}^{2}} u=Ke^{2u}-2\pi\sum_{i=1}^{m} \beta_i\delta_{p_i}-1
		\end{align*}
		has no solutions in $\cap_{p\in [1,2)} W^{1,p}(\mathbb{S}^{2},g_{\mathbb{S}^{2}})$, then there exists a neighbourhood $\mathscr{U}$ of  $(K,\beta)$ in $C^{1}(\mathbb{S}^{2}) \times (-1,\infty)^{\times m}$ such that for any $( \tilde{K},\tilde{\beta} ) \in \mathscr{U}$, 
		\begin{align*}
		-\Delta_{\mathbb{S}^{2}} u=\tilde{K} e^{2u}-2\pi\sum_{i=1}^{m} \tilde{\beta}_i\delta_{p_i}-1
		\end{align*}
		has no solutions in $\cap_{p\in [1,2)} W^{1,p}(\mathbb{S}^{2}, g_{\mathbb{S}^{2}})$. 
	\end{theorem}

	A key ingredient in proving Theorem \ref{thm  openness of nonexistence K C0} is the compactness statement in Theorem \ref{theorem metric convergence on S2} below. Fix $m$ distinct points $p_1,...,p_m$ on $\mathbb S^2$. Let $\beta^k=(\beta^k_1,...,\beta^k_m)\in\mathbb R^m$ with $\beta^k\to\beta\in\mathbb R^m$ as $k\to\infty$ and ${\mathfrak D}^k=\beta_1^{k} {p_1}+\cdots+\beta_m^{k} {p_m}$. For a sequence of conformal metrics $g_k=e^{2u_k}g_{\mathbb S^2}$ representing ${\mathfrak D}^k$ with curvature $K_k\to K>0$, the theorem says that (up to passing to subsequences) either the sequence converges to a conformal metric with curvature $K$ representing a divisor $\mathfrak D=\sum^m_{i=1}(\beta_i -2n_i)p_i$ where $n_i\in\mathbb N\cup\{0\}$, or the sequence collapses to 0-measure but after normalization it converges to a conformal metric representing a nontrivial divisor (possibly different from $\mathfrak D$) on $\mathbb S^2$ with curvature 0. 
	
	Suppose that the curvature measures $\K_{g_k}$ of $g_k$ (or a subsequence under consideration) converge weakly. 
	The function 
	$\Theta(x):\mathbb{S}^{2} \rightarrow \mathbb{R}$ defined by 
	$$
	\Theta(x )=\lim_{r \rightarrow 0} \lim_{k \rightarrow \infty} \mathbb{K}_{g_{k}}( B_{r}^{g_{\mathbb{S}^{2}}}(x)  ). 
	$$
	reveals the curvature concentration at $x$ and we can show that
	$\{x: \Theta(x) \neq 0 \}$ has only finite elements, hence the sum $\sum_{x \in \mathbb{S}^{2}} \Theta(x) \delta_{x}$ is well-defined. 
	The value of $\Theta(x)$ can be calculated precisely by analyzing the bubble tree structure carefully (see Proposition \ref{pro theta's value when no bubble}, Proposition \ref{pro Theta's value when uk converges} and Proposition \ref{pro v's equation when uk not converge}). 
	
	\begin{theorem}
		\label{theorem metric convergence on S2}
		Let $p_{1}, ... , p_{m}$ be distinct points on $\mathbb S^2$. Suppose that $g_k=e^{2u_k}g_{\mathbb{S}^2}$ and $u_{k}$ satisfies 
		\begin{align*}
		-\Delta_{\mathbb{S}^{2}} u_{k}=K_{k} e^{2u_{k}}-1-2 \pi \sum_{i=1}^{m} \beta_{i}^{k} \delta_{p_{i}},
		\end{align*}
		in the sense of distribution, where 
		\begin{itemize}
			\item[(A1)] $K_k\in C^0(\mathbb S^2)$ and $K_k\rightarrow K>0$ in $C^0(\mathbb{S}^2)$,
			\item[(A2)] $\beta_i^{k} \to \beta_i>-1$,
			\item[(A3)] $\chi(\mathbb{S}^2,\beta)=2+\sum\limits_{i=1}^m\beta_i>0$.
		\end{itemize}
		Then after passing to a subsequence 
		one of the following holds:
		\begin{itemize}
			\item[(a) ] If $u_{k} \rightarrow u$ weakly in $\cap_{p\in{(1,2)}}W^{1,p}(\mathbb{S}^{2},g_{\mathbb{S}^{2}})$, 
			then $u$ solves
			\begin{align*}
			-\Delta_{\mathbb{S}^{2}} u=K e^{2u}+\sum_{x \in \mathbb{S}^{2}} \Theta(x) \delta_{x}-1. 
			\end{align*}
			$\{u_{k}\}$ can only has bubbles at $p_{i}$ with $\beta_{i}>1$, and the bubble tree at each such point has  only one level. Furthermore, the total number of bubbles, saying $s$, has an upper bound:
			$$
			s \leq \frac{1}{2}\chi(S^2,\beta). 
			$$

			\item[(b)]  If $u_{k}-c_{k} \rightarrow v$ weakly in $\cap_{p\in{(1,2)}}W^{1,p}(\mathbb{S}^{2},g_{\mathbb{S}^{2}})$ with $c_{k} \rightarrow -\infty $, 
			then 
			\begin{align*}
			-\Delta_{\mathbb{S}^{2}} v=\sum_{x \in \mathbb{S}^{2}} \Theta(x) \delta_{x}-1.
			\end{align*} 
			Each bubble tree has at most two levels, and 
			$$
			s_{1}-s_{2}=\frac{1}{2}\chi(S^2,\beta)-\sum_{ \{ i: \{u_k\}\text{ has a singular bubble at }p_i \}  }\beta_i,
			$$
			where $s_{1}$ and $s_{2}$ denote the total number of bubbles of $\{u_{k}\}$ at $1$-level and $2$-level, respectively. 

		\end{itemize}
	\end{theorem}

	\vspace{.1cm}
	We will show that (a) in the above theorem does not happen and the bubble trees have only one level, 
	provided one of the following holds:
	\begin{enumerate}
		\item  $\chi(\mathbb{S}^{2},\beta)<2$
		\item $\beta_{i} \leq 1$ for any $i \in \{1,\cdots,m\}$.
	\end{enumerate}
	For these two cases, the number of bubbles is bounded by the divisor:
	$$
	\#\{\text{bubbles}\}\leq 1+\frac{1}{2}\sum_{i=1}^{m} |\beta_i|.
	$$
	
	Convergence of solutions without divisors has been studied extensively, for example, in \cite{Brezis1991, li-Shafrir1994, li1999harnack, ding1999existence, Tarantello2005}, etc. For conformal metrics representing divisors, it is shown in \cite{bartolucci2002, Bartolucci2007} that if $\{u_{k}\}$ has at least one bubble and $K_{k} \rightarrow K$ in $C^{1}(\mathbb{S}^{2})$ (compare to $C^0(\mathbb S^2)$ in Theorem \ref{theorem metric convergence on S2}), then (b) in Theorem \ref{theorem metric convergence on S2} holds and $v$ solves 
	\begin{align}
	-\Delta_{\mathbb{S}^{2}} v=\sum_{j=1}^{s_{1}} ( 4 \pi +4 \pi \beta_{i_{j}}) \delta_{p_{i_{j}}}+\sum_{j=1}^{s_{2}} 4 \pi \delta_{q_{j}}-2 \pi \sum_{i=1}^{m} \beta_{i} \delta_{p_{i}}-1. 
	\end{align}	
	
	Using the techniques developed in this paper, we can give a detailed description on the bubble development under the stronger assumption $K_k\to K$ in $C^1(\mathbb S^2)$:  
	the sequence $\{u_{k}\}$ can only develop bubbles at 1-level, $u_k\rightarrow-\infty$ almost everywhere, and the bubble tree structure can be described in a clear way (see the detailed argument in Proposition \ref{pro v's equation when Kk converges in C1}):
	\begin{enumerate}
		\item $\{u_{k}\}$ has $s_{i}$ smooth bubbles at 1-level at some $p_{i}$ with $s_{i}=\beta_{i}+1$; thus, if $\beta_i\notin\mathbb{N}$, then $\{u_{k}\}$ has no  smooth bubbles at $p_{i}$; 
		\item $\{u_{k}\}$ has one bubble with two singularities at 1-level at some $p_{i}$; 
		\item $\{u_{k}\}$ has one smooth bubble at 1-level at some $q_{j} \notin \{p_{1}, \cdots,p_{m}\}$.
	\end{enumerate}
	
	\vspace{.1cm}
	
	Finally, we would like to mention that Mazzeo-Zhu have developed a theory for moduli of metrics with divisors in \cite{Mazzeo-Zhu} (see the reference therein).

	
	
	\medskip
	
	\noindent{\bf Acknowledgement.} The first author would like to thank Professor Gang Tian for arranging the visits to BICMR in the summers of 2023 and 2024 in the course of this work. All authors are grateful for the wonderful research environment provided by Tsinghua University during their collaboration.

	\section{Preliminary}
	In this section, we collect a few definitions and results from \cite{C-L1,C-L2} and \cite{CLWnegative} that will be used in this paper.

	Let $\Sigma$ be a closed surface with a Riemannian metric $g_0$, the Gauss curvature $K_{g_0}$ and the area element $dV_{g_0}$. 
	Let $\M(\Sigma,g_0)$ be the set of Radon measures $g_u=e^{2u}g_0$ so that there is a {\it signed} Radon measure $\mu(g_u)$ for $u\in L^1(\Sigma,g_0)$ satisfying 
	\begin{equation*}
	\int_\Sigma\varphi \,d\mu(g_u)
	=\int_\Sigma \left(\varphi \,K_{g_0} - u \Delta_{g_0}\varphi\right)dV_{g_0},\s \forall \varphi\in C^\infty_0(\Sigma).
	\end{equation*}
	We write $dV_{g_u}= e^{2u} dV_{g_0}$
	and call the signed Radon measure $\K_{g_u}= \mu(g_u)$ the curvature measure for $g_u$. 
	
	In an isothermal coordinate chart $(x,y)$ for the smooth metric $g_0$, we can write $g_0= e^{2u_0}g_{\euc}$ for some local function $u_0$. So any $g\in\M(\Sigma,g_0)$ is locally expressible as $g=e^{2v}g_{\euc}$, where $v\in L^{1}_{\loc}(\Sigma)$
	and
	\begin{equation}
	-\Delta v \ dxdy =\K_{g_v}
	\end{equation}
	as distributions where $\Delta=\frac{\partial^2}{\partial x^2}+\frac{\partial^2}{\partial y^2}$.
	
	\begin{defi}\label{def represent divisors in distributions}\textnormal{
			Let $(\Sigma,g_0)$ be a surface and $\mathfrak{D}=\sum_{i=1}^m \beta_{i} p_{i}$ be a divisor on $\Sigma$. A Radon measure $g=e^{2u}g_0\in\mathcal{M}(\Sigma,g_0)$ {\it represents $\mathfrak{D}$ with curvature function $K$} if
			$$
			\K_{g}=K e^{2u} dV_{g_{0}} -2\pi\sum_{i=1}^{m} \beta_i\delta_{p_i},\s and\s K e^{2u}\in L^1(\Sigma,g_{0}).
			$$
			In other words, in the sense of distributions, 
			\begin{equation}\label{Gauss.equation}
			-\Delta_{g_0}u=K e^{2u}-K_{g_0}-2\pi\sum_{i=1}^{m} \beta_i\delta_{p_i}.
			\end{equation}}
	\end{defi}
	
	By \eqref{Gauss.equation}, for a closed surface $\Sigma$ we have 
	$$
	\frac{1}{2\pi}\int_{\Sigma}K e^{2u} dV_{g_{0}}=\chi(\Sigma)+\sum_{i=1}^{m}  \beta_i=\chi(\Sigma,\beta).
	$$
	
	For simplicity, we use the notations: $\mathcal{M}(\Omega)=\mathcal{M}( \Omega,g_{\euc}  )$ where $\Omega \subset \mathbb{R}^2$ is a domain and 
	$\mathcal{M}(\mathbb{S}^{2})=\mathcal{M}(\mathbb{S}^{2}, g_{\mathbb{S}^{2}})$ where $g_{\mathbb{S}^{2}}$ is the metric on $\mathbb{S}^{2}$ of curvature 1.  

	\begin{thm}\cite{C-L2}\label{thm convergence of solutions of Radon measure when measure is small}
		Let $\{ \mu_k \}$ be a sequence of signed Radon measures on $D$ with $|\mu_k|(D)<\epsilon_0< \pi$. Suppose that $-\Delta u_k=\mu_k$ holds weakly with $\|\nabla u_k\|_{L^1(D)}<\Lambda$ and $\Area(D,g_k)<\Lambda'$. Then after passing to a subsequence, one of the following holds:
		\begin{itemize}
			\item[(1)]  $u_k\to u$ weakly in $W^{1,p}(D_{{1}/{2}})$ for any $p \in [1,2)$ and  $e^{2 u_k}\to e^{2 u}$ in $L^q(D_{{1}/{2}})$ for some $q>1$;
			\item[(2)] $u_k\to -\infty$ for a.e. $x$ and $e^{2u_k}\to 0$ in $L^q(D_{1/2})$ for some $q>1$.
		\end{itemize}
	\end{thm}

	\begin{lem}\cite{C-L2}
		\label{lemma gradient estimate on Riemann surfaces}
		Let $\mu$ be a signed Radon measure defined on a closed surface $(\Sigma,g_{0})$ and $u \in L^{1}(\Sigma,g_{0})$ solves $-\Delta_{g_{0}} u= \mu$ weakly. Then, for any $r>0$ and $q \in [1,2)$, there exists $C=C(q,r,g_{0})$ such that 
		\begin{align*}
		r^{q-2} \int_{ B_{r}^{g_{0}}(x)} |\nabla_{g_{0}} u |^{q} \leq C (|\mu|(\Sigma) )^{q}. 
		\end{align*}
	\end{lem}

	We now discuss terminologies needed to describe the {\it bubble-tree} analysis. 
	Let $g_k=e^{2u_k}g_{\euc}\in\mathcal{M}(D)$ with $|\K_{g_k}|(D)<\Lambda$. 
	
	\begin{defi}
		\label{def blowup sequence and bubble}\textnormal{
			A sequence $\{ (x_k,r_k) \}$ is a {\it blowup sequence of $\{u_{k}\}$ at $0$} if $x_k\rightarrow 0$, $r_k\rightarrow 0$ and $u_k'=u_k(x_k+r_kx)+\log r_k$ converges weakly to a function $u$ in $\cap_{p\in[1,2)}W^{1,p}_{\loc}(\R^2)$, and $u$ is called a {\it bubble}. The sequence $\{u_{k}\}$ {\it has no bubble at $0$} if no subsequence of $\{u_{k}\}$ has a blowup sequence at $0$. We say $u_k'$ converges to a {\it ghost bubble} if there exists $c_{k} \rightarrow -\infty$ such that $u_k'-c_{k}$ converges weakly to some function $v$ in $\cap_{p\in[1,2)}W^{1,p}_{\loc}(\R^2)$. }
	\end{defi}

	The ghost bubble has the following property:
	\begin{align*}
	\lim_{r \rightarrow 0} \lim_{k \rightarrow \infty} \int_{ D_{{1}/{r} } \setminus \cup_{z \in \mathcal{S} } D_{r}(z)} e^{ 2u_k'  } =0,
	\end{align*}
	where $\mathcal{S}$ is the set consisting of measure-concentration points:
	$$
	\S= \left\{y: \mu(y) \geq \frac{\epsilon_{0}}{2}     \right\}.
	$$
	where $\mu$ is the weak limit of $|\mathbb{K}_{ e^{ 2 (u_k'-c_k) } g_{\euc} }|$ and $\epsilon_{0}$ is chosen as in Theorem \ref{thm convergence of solutions of Radon measure when measure is small}.
	
	At a point, the sequence $\{u_k\}$ may have more than one blowup sequences, we distinguish them according to the following definitions.
	
	\begin{defi}
		\label{def bubble different}
		\textnormal{Two blowup sequences $\{(x_k,r_k)\}$ and $\{(x_k',r_k')\}$ of $\{u_k\}$ at $0$ are {\it essentially different} if one of the following happens
			$$
			\frac{r_k}{r_k'}\rightarrow\infty,\mbox{ or }\,
			\frac{r_k'}{r_k}\rightarrow\infty,
			\s \mbox{or} \s\frac{|x_k-x_k'|}{r_k+r_k'}
			\rightarrow\infty.
			$$
			Otherwise, they are {\it essentially same}. }
	\end{defi}
	
	\begin{defi}
		\textnormal{We say {\it the sequence  $\{u_k\}$ has $m$ bubbles} if $\{u_k\}$ has $m$ essentially different blowup sequences and no subsequence of $\{u_k\}$ has more than $m$ essentially different blowup sequences.}
	\end{defi}
	
	\begin{defi}
		\label{def bubble on top}
		\textnormal{For two essentially different blowup sequences $\{(x_k,r_k)\}, \{(x_k',r_k')\}$, we say $\{ (x_k',r_k') \}$ is {\it on the top of $\{ (x_k,r_k) \}$} and write $\{ (x_k',r_k')\}<\{ (x_k,r_k)\}$, if
			$\frac{r_k'}{r_k}\rightarrow 0$ and
			$\frac{x_k'-x_k}{r_k}$ converges as $k\to\infty$. }
	\end{defi}
	
	\begin{rem}\label{remark different bubble relationship}\textnormal{
			(1) If $\{ (x_k',r_k') \}<\{ (x_k,r_k) \}$, we have
			$
			l_k:=\frac{r_k'}{r_k}\rightarrow 0$ and $
			y_{k}:=\frac{x_k'-x_k}{r_k}\rightarrow y.
			$
			If we set $v_k=u_k(x_k+r_kx)+\log r_k$ and $v_k'=u_k(x'_k+r_k'x)+\log r_k'$, we can verify $
			v_k'(x)=v_k \left(l_kx+y_k\right)+\log l_{k}$. 
			Then $\{ (y_k,l_k) \}$ is a blowup sequence of $\{ v_k \}$ at $y$, and the limit of $v_k'$ can be considered as a bubble of $\{  v_k \}$.\vspace{.1cm}
			\newline 
			(2) If two essentially different blowup sequences $\{ (x_k,r_k)\}, \{ (x_k',r_k') \}$ are not on the top of each other, then we must have
			$
			\frac{|x_k-x_k'|}{r_k+r_k'} \rightarrow \infty 
			$
			and separation of domains: for any $t>0$, when $k$ is sufficiently large it holds 
			$
			D_{tr_k}(x_k)\cap D_{tr_k'}(x_k')=\emptyset
			$.}
	\end{rem}
	
	\begin{defi}
		\label{def bubble right on top}
		\textnormal{$\{ (x_k',r_k') \}$ is {\it right on the top} of $\{ (x_k,r_k) \}$, if $\{ (x_k',r_k') \}<\{ (x_k,r_k) \}$ and there is no $\{ (x_k'',r_k'') \}$ with
			$
			\{ (x_k',r_k') \}<\{ (x_k'',r_k'') \}<\{ (x_k,r_k) \}.
			$
		}
	\end{defi}
	We define the {\it level of a blowup sequence} and its corresponding bubble as follows:
	
	\begin{defi}
		\label{def bubble level}
		\textnormal{			\begin{enumerate}
				\item A blowup sequence $\{ (x_{k},r_{k}) \} $ of $\{ u_{k} \}$ at the point 0 is {\it at $1$-level}  if there is no $\{ (x_k',r_k') \}$ with $\{ (x_k,r_k) \} < \{ (x_k',r_k')\}$. 
				For $m \geq 2$, a blowup sequence $\{ (x_{k},r_{k}) \} $ is {\it at $m$-level} if there exists a blowup sequence $\{ (x_k',r_k') \} $ which is at $(m-1)$-level such that $\{ (x_{k},r_{k}) \} <\{ (x_k',r_k') \} $. 	
				\item For $m \geq 1$, {\it a bubble is at $m$-level }if it is induced by a blowup sequence  at $m$-level.
		\end{enumerate} }
	\end{defi}
	
	Two essentially different blowup sequences at $m$-level may be  right on the top of two essentially different blowup sequences at $(m-1)$-level. 
	By Remark \ref{remark different bubble relationship}, if $\{ (x_k^1,r_k^1) \},\{ (x_k^2,r_k^2) \}$ are at the same level, then  for any $t>0$, when $k$ is large 
	$$
	D_{tr_k^1}(x_k^1)\cap D_{tr_k^2}(x_k^2) =\emptyset
	$$
	
	Let $ g=e^{2u}g_{\euc} \in \mathcal{M}( D \backslash \{0\})$ with 
	\begin{align*}
	|\mathbb{K}_{g}|( D \backslash \{0\}  ) <\epsilon_{1}<\pi.  
	\end{align*}
	for some $\epsilon_1$. Extend $\K_g$ to a signed Radon measure $\mu$ by taking $\mu(A)=\mathbb{K}_{g}(A \cap (D \backslash \{0\})), \forall A \subset \mathbb{R}^{2}$ and write
	$\mu=\K_g\lfloor (D\backslash\{0\})$.
	By \cite{CLWnegative}, we have the following decomposition:
	\begin{align*}
	u(z)=I_{\mu}(z)+ \lambda \log|z|+w(z), 
	\end{align*}
	where
	\begin{align*}
	I_{\mu}(z)=-\frac{1}{2 \pi} \int_{\mathbb{R}^{2}} \log |z-y| d \mu(y), \quad \lambda=\lim_{{\rm a.e}\,\,r\rightarrow 0}\frac{1}{2 \pi }\int_{\partial D_r}\frac{\partial u}{\partial r},
	\end{align*}
	and $w$ is a  harmonic function on $D\backslash \{0\}$ with $\int_{\partial D_t}\frac{\partial w}{\partial t}=0$. Hence, we can find a holomorphic function $F$ on $D\backslash\{0\}$ with
	$\Re(F)=w$, which means
	\begin{align*}
	(u-I_{\mu})(z)=\Re(F(z))+\lambda \log |z|, \quad z \in D \backslash \{0\}.
	\end{align*}
	
	\begin{lem}\cite{CLWnegative}
		\label{lemma removable singularity}
		\textnormal{(1)} If the area of $D$ is finite, namely,
		\begin{align*}
		\int_{D} e^{2u}<\infty,
		\end{align*}
		then $w$ is smooth on $D$, and $\lambda \geq -1$. 
		Moreover, $g\in \mathcal{M}(D)$ with
		$$
		\K_g=\mu-2\pi\lambda\delta_0.
		$$
		\textnormal{(2)} If we further assume $\mathbb{K}_{g} \geq 0$ on $D \backslash \{0\}$, then $\lambda>-1$. 
	\end{lem}

	\begin{defi}
		\label{def residue}\textnormal{
			(1)
			For $g=e^{2u}g_{\euc}\in\mathcal{M}(D)$, the {\it residue} of $g$ (or $u$) at $0$ is 
			$$
			\res(g,0)=\res(u,0)=-\frac{1}{2\pi}\K_g(\{0\}).
			$$
		}
		\textnormal{(2) For $g=e^{2u}g_{\euc}\in\mathcal{M}(D \backslash \{0\})$ satisfying
			\begin{align*}
			|\mathbb{K}_{g}|(D \backslash \{0\})+\int_{D} e^{2u}<\infty,
			\end{align*}
			the residue of $g$ (or $u$) at $0$ is 
			$$
			\res(g,0)=\res(u,0)=-\frac{1}{2\pi}\K_g(\{0\}).
			$$
		}
		\textnormal{ (3) 
			For  $g=e^{2u}g_{\euc}\in\mathcal{M}(\R^2\backslash D)$ satisfying
			\begin{align*}
			|\K_g|(\R^2\backslash D)+\int_{ \R^2\backslash D } e^{2u}<\infty,
			\end{align*}
			the residue of $g$ (or $u$) at $\infty$ is 			
			\begin{align*}
			\res(g,\infty)=\res(u,\infty)=-\frac{1}{2\pi}\K_{g'}(\{0\})
			\end{align*}
			where $g'=e^{2u'}g_{\euc}\in\mathcal{M}(D\backslash\{0\})$ with 
			$u'(x')=u(\frac{x'}{|x'|^2})-2\log|x'|, x'=\frac{x}{|x|^2}$
			is extended in $\mathcal{M}(D)$. 
		}
	\end{defi}

	\begin{pro}\cite{CLWnegative}
		\label{prop non-positive on neck} 
		Let $g_k=e^{2u_k}g_{\euc}\in\mathcal{M}(D)$ with
		\begin{align*}
		|\K_{g_k}|(D)+\Area(D,g_k)\leq\Lambda.
		\end{align*}
		Assume that $u_k$ converges to $u$ weakly in $W^{1,p}(D)$ for some $p\in[1,2)$
		and $\{u_k\}$ has a blowup sequence $\{(x_k,r_k)\}$ at $0$ with the corresponding bubble $u'$. Then there exists $t_i\rightarrow 0$, such that
		\begin{align*}
		\lim_{i\rightarrow\infty}\lim_{k\rightarrow\infty}\K_{g_k}(D_{t_i}(0)\backslash D_{ \frac{r_k}{t_i}}  (x_k))=-2 \pi \left(2+\Res(u,0)+\Res(u',\infty)\right).
		\end{align*}
	\end{pro}

	If no bubble occurs then there is no area concentration point: 
	
	\begin{pro}\cite{CLWnegative}
		\label{pro no area concentration at zero on D}
		Let $g_{k}=e^{2u_{k}} g_{\euc} \in \mathcal{M}(D)$ which satisfies \textnormal{(P1)-(P3)}.
		Assume that $\{u_{k}\}$ has no bubble. Then 
		\begin{align}
		\label{identity.disk}
		\lim_{r \rightarrow 0} \lim_{k \rightarrow \infty} \Area( D_{r}, g_{k})=0. 
		\end{align}
	\end{pro}
	
	If the curvature does not change sign then convergence of solutions occurs: 
	
	\begin{pro}\cite{CLWnegative}
		\label{prop convergence.with singular point}
		Let $g_k=e^{2u_k}g_{\euc}\in\mathcal{M}(D)$ with $\K_{g_k}=f_ke^{2u_k}dx-2\pi\beta_k\delta_0$. Assume 
		\begin{itemize}
			\item[(1)] $\Area(D,g_k)\leq\Lambda_1$;
			
			\item[(2)] \( r \|\nabla u_{k}\|_{L^{1}(D_r(x))} \leq \Lambda_{2} \) for all \( D_r(x) \subset D  \);
			
			\item[(3)] $-\Lambda_3 \leq f_k\leq -1$ or $1\leq f_k\leq \Lambda_3$,  and $f_k$ converges to $f$ for a.e. $x\in D$;
			
			\item[(4)] $\beta_k\rightarrow\beta$.
		\end{itemize}
		Assume $\{ u_k \}$ has no bubble. Then, after passing to a subsequence, one of the following holds:
		\begin{itemize}
			\item[(a)] $u_k\to u$ weakly in $\cap_{p \in [1,2)} W^{1,p}(D_{1/2})$
			and $e^{2u_k}\to e^{2u}$ in $L^1(D_{1/2})$, and 
			$$
			-\Delta u=fe^{2u}-2\pi \beta\delta_0;
			$$
			\item[(b)] $u_k\to -\infty$ a.e. and $e^{2u_k}\to 0$ in $L^1(D_{1/2})$.
		\end{itemize}
	\end{pro}
	
	The following {\it area identity} plays a key role in this paper: 
	
	\begin{theorem}\cite{CLWnegative}
		\label{theorem area convergence of bubble tree}
		Let $g_{k} =e^{2u_{k}} g_{\euc} \in \mathcal{M}(D)$ satisfy 
		\begin{itemize}
			\item[(P1)] \( \K_{g_k} = f_k dV_{g_k} + \lambda_k \delta_0 \) (where \( \lambda_k \) might be 0) with \( f_k \geq 1 \) or \( f_k \leq -1 \);
			\item[(P2)] \( |\K_{g_k}|(D) + \Area(D,g_k) \leq  \Lambda_{1} \);
			\item[(P3)] \( r^{-1} \|\nabla u_{k}\|_{L^{1}(D_r(x))} \leq \Lambda_{2} \) for all \( D_r(x) \subset D \).
		\end{itemize}
		We assume $\{u_{k}\}$ has finitely many bubbles $v_{1}, \cdots, v_{m}$, induced by blowup sequences $\{(x_{k}^{1},r_{k}^{1}) \}$, $\cdots, \{ (x_{k}^{m},r_{k}^{m})\}$ at the point $0$, respectively. and we further assume that $\{u_{k}\}$ has no bubbles at any point $x \in D \backslash \{0\}$. Then after passing to a subsequence, 
		\begin{align}
		\label{bubble tree area convergence}
		\lim_{k \rightarrow \infty} \Area( D_{{1}/{2}},g_{k})=\Area(D_{{1}/{2}},g_{\infty})+\sum_{i=1}^m  \Area(\mathbb{R}^{2},e^{2v_{i}}g_{\euc}).
		\end{align}
		We allow $m=0$, which means $\{u_{k}\}$ has no bubbles at $0$, and the sum term in 	(\ref{bubble tree area convergence}) vanishes for this case. 
	\end{theorem}

	\section{Positive curvatures on $\mathbb S^2$ with divisors}
	
	In this section, we investigate the bubble tree convergence of a sequence of conformal metrics with conical singularities on a sphere. 
	
	\subsection{Bubbles at 1-level}
	
	In this subsection, we establish a result concerning a property of bubbles at 1-level on a disk, 
	which will play a crucial role in the subsequent discussions. 
	\begin{lem}
		\label{lemma one.level} 
		Let $g_{k}=e^{2u_{k}} g_{\euc }\in \mathcal{M}(D)$ and let $u_k\in \bigcap_{p\in[1,2)}W^{1,p}(D)$ solve 
		$$
		-\Delta u_k = K_k(x)e^{2u_{k}} - 2\pi\lambda_k\delta_{y_k},
		$$
		under the following assumptions:
		\begin{itemize}
			\item[(a)] There exists $\Lambda_{1} >0$ such that $\Area(D,g_{k}) \leq \Lambda_{1}$;
			\item[(b)] There exists $\Lambda_{2} >0$ such that $ r^{-1} \|\nabla u_{k}\|_{L^{1}(D_r(x))} \leq \Lambda_{2} \text{ for all }  D_r(x) \subset D  $;
			\item[(c)] $\{u_k\}$ converges to $u$ weakly in $\bigcap_{p\in[1,2)} W^{1,p}(D)$;
			\item[(d)] $K_k$ converges to $K$ in   $C(\overline{D})$, where $K>0$;
			\item[(e)] The sequences $\{y_k\}$ and $\{\lambda_k\}$ converge to $y$ and $\lambda \in(-1,\infty)$, respectively.
		\end{itemize}
		Then, we have:
		\begin{itemize}
			\item[(1)] $\{u_k\}$ converges weakly in $W^{2,p}_{\text{loc}}(D\backslash\{y\})$ for any $p>1$; 
			
			\item[(2)] If  $\{ u_k \}$ has a blowup sequence $\{ (x_k,r_k) \} $, then $x_{k} \rightarrow y$, $y_k\in D_t(y)\backslash D_{\frac{r_k}{t}}(x_k)$ for any fixed $t>0$ and large $k$, and the corresponding bubble is smooth;
			
			\item[(3)] If $\lambda \leq 1$, then $\{ u_k \}$ has no bubble in $D$.
		\end{itemize}
	\end{lem} 
	
	\begin{proof}
		\noindent	\textbf{Step 1:} We show that $u_k$ converges weakly in $W^{2,p}_{\text{loc}}(D\backslash\{y\})$ for any $p>1$. 
		
		We assume that $|\K_{g_k}|$ converges weakly to a Radon measure $\mu$  and define 
		$$
		\S=\left\{x:\mu(\{x\})>\frac{\epsilon_0}{2}\right\},
		$$ 
		where $\epsilon_{0}$ is chosen as in Theorem \ref{thm convergence of solutions of Radon measure when measure is small}.
		
		By Theorem \ref{thm convergence of solutions of Radon measure when measure is small},
		$e^{2u_k}$ is bounded in $L^{q_1}(\Omega)$ for some $q_1>1$ and any $\Omega\subset\subset D\backslash\S$. By the $L^p$-estimate \cite[Theorem 9.11]{gilbarg1977elliptic}, $u_k\to u$ weakly in $W^{2,p}_{\text{loc}}(D\backslash\S)$ for any $p \geq 1 $. Thus, to see $u_k\to u$ in $W^{2,p}_{\text{loc}}(D\backslash\{y\})$, it suffices to show $\S\backslash\{y\}=\emptyset$.
		
		We assume that there exists $\tilde{y} \in\S\backslash\{y\}$, and select $\delta$ such that $D_{2\delta}(\tilde{y})\cap(\S\cup\{y\})=\{\tilde{y}\}$. For large $k$, the equation satisfied by $u_k$ on $D_{2\delta}(\tilde{y} )$ is 
		\begin{equation*}
		\label{equation.uk' one level}
		-\Delta u_k=K_ke^{2u_k}.
		\end{equation*}
		If $u_k$ is bounded from above in $D_{2\delta}(\tilde{y} )$, then $u_k$ is bounded in $W^{2,p}(\overline{D_{\delta}(\tilde{y} )})$ for any $p$. This implies that $K_ke^{2u_k}$ converges in $C((\overline{D_{\delta}(\tilde{y} )})$, which contradicts the choice that $\tilde{y}$ is in $\S$. Thus, we must have $\max_{\overline{D_\delta(\tilde{y} ) } } u_k\rightarrow +\infty$. Let
		$$
		c_k=\max_{\overline{D_\delta(\tilde{y} )} }u_k=u_k(z_k),\s \rho_k=e^{-c_k},
		$$
		and define
		$$
		\tilde{u}_k(x)=u_k(z_k+\rho_kx)+\log \rho_k.
		$$
		
		Recalling that $u_k\to u$ weakly in $W^{2,p}_{\text{loc}}(D_{2\delta}\backslash\{ \tilde{y}  \})$ for any $p$, we conclude $z_k\rightarrow \tilde{y} $ (otherwise we will obtain a contradiction to the choice of $z_{k}$). 
		Since $\tilde{u}_k(0)=\max_{D_R} \tilde{u}_k$, applying \cite[Theorem 0]{li-Shafrir1994} and \cite[Theorem 9.11]{gilbarg1977elliptic}, we conclude $\|\tilde{u}_k\|_{W^{2,p}(D_R)}<C(R)$ for any $R$. Then, $\tilde{u}_k$ converges weakly in $W_{loc}^{2,p}(\mathbb{R}^{2})$ to a function $\tilde{u}$ which satisfies 
		$$
		-\Delta \tilde{u}=K(\tilde{y})e^{2\tilde{u}},\s \tilde{u}(0)=\max_{\mathbb{R}^{2}} \tilde{u}=0,\s \int_{\R^2}e^{2\tilde{u}}<\infty.
		$$ 
		By \cite[Theorem 1]{CW-L}, we have
		$$
		\tilde{u}(x)=-\log\left(1+\frac{K(\tilde{y} )}{4}|x|^2\right),\s \textnormal{and}\s \int_{\R^2}K(\tilde{y} )e^{2\tilde{u}}=4\pi.
		$$
		Thus, $e^{2\tilde{u}}g_{{\euc}}$ is the metric with constant curvature $K(\tilde{y} )$ on $\mathbb{S}^2\backslash\{\infty\}$. So $\infty$ is not a singularity hence 
		$\res(\tilde{u},\infty)=0$. By Proposition \ref{prop non-positive on neck}, we obtain
		$$
		\liminf_{r\rightarrow 0}\lim_{k\rightarrow\infty}
		\K_{g_k}(D_r(\tilde{y} )\backslash D_{\frac{r_k}{r}}(z_k)) \leq -2 \pi \left(2+\Res(u,\tilde{y})+\Res(\tilde{u},\infty)\right)<0.
		$$
		However, since $\K_{g_k}=K_ke^{2 u_k}dx$ on $D_r(\tilde{y} )\backslash D_{\frac{r_k}{r}}(z_k)$ and $K_k>0$, we arrive at a contradiction.		
		
		\vspace{.1cm}
		\noindent	\textbf{Step 2:} 
		Assume that $\{u_{k}\}$ has a  blowup sequence $\{(x_{k},r_{k})\}$ at some point $x_{0} \in D$. We claim that for any fixed $r>0$, $y_{k}$ lies in the neck region $D_{r}(x_{0}) \backslash D_{ \frac{r_{k}}{r} }(x_{k})$. 
		
		Suppose $u_k'=u_k(x_k+r_kx)+\log r_k\to u'$ weakly in $\cap_{p \in [1,2)} W_{\loc}^{1,p}(\mathbb{R}^{n})$. If $y_k\notin D_r(x_{0})\backslash D_{\frac{r_k}{r}}(x_k)$,
		then 
		$u_k$ has no singularity in $D_r(x_{0})\backslash D_{\frac{r_k}{r}}(x_k)$. By Proposition \ref{prop non-positive on neck},
		\begin{align*}
		0\leq\liminf_{r\rightarrow 0}\liminf_{k\rightarrow\infty}\K_{g_k}(D_r(x_{0})\backslash D_{\frac{r_k}{r}}(x_k))\leq-2\pi(2+\res(u,x_{0})+\res(u',\infty)),
		\end{align*}
		which  leads to a contradiction since $\res(u,x_{0})>-1$ and $\res(u',\infty)>-1$. 
		
		Consequently, $x_{0}=y$, $|-\frac{y_k-x_k}{r_k} |  \rightarrow\infty$. In turn, for any $r>0$, $u_k'$ solves the following equation on $D_{r}$ when $k$ is sufficiently large:
		\begin{align*}
		-\Delta u_{k}'=K_{k}( x_{k}+r_{k} x ) e^{2u_k'}.
		\end{align*}
		Applying similar argument as in \textbf{Step 1}, $u_k'\to u'$  weakly in $W^{2,p}_{\text{loc}}(\mathbb{R}^2)$ for any $p>1$, where $u'$ satisfies 
		\begin{align*}
		-\Delta u'=K(x_0)e^{2u'},\s \int_{\R^2}e^{2u'}<\infty.
		\end{align*}
		
		Applying the result in \cite{CW-L} again, we deduce that $\res(u',\infty)=0$. 
		\vspace{.1cm}
		
		By Proposition \ref{prop non-positive on neck},
		\begin{align*}
		-2\pi\lambda &\leq\liminf_{r\rightarrow 0}\liminf_{k\rightarrow\infty}\K_{g_k}(D_r(y)\backslash D_{\frac{r_k}{r}}(x_k))\\&\leq-2\pi\big(2+\res(u,y)+\res(u',\infty)\big)<-2 \pi 
		\end{align*}
		which implies that $\lambda>1$. 
		
		Therefore, if $\lambda \leq 1$, $\{u_{k}\}$ has no bubble in $D$. 
	\end{proof}
	
	
	\begin{rem}	\textnormal{
			Synge's theorem tells us that a Riemannian surface with positive Gauss curvature is simply-connected and has no closed minimizing geodesics. Heuristically, Lemma \ref{lemma one.level} (2) can be viewed as a version of Synge's theorem on surfaces with positive curvature measure: if (2) does not hold, the neck which joins the bubble part and the base part has no singularities so it will eventually disappear in the blowup process, and it 
			looks like the middle part of a dumbbell, which is ``thin". Hence a ``closed minimizing geodesic" will occur, see Figure I. }
	\end{rem}
	\begin{center}
		
		\scalebox{.6}{\begin{tikzpicture}
			
			\draw[domain=30:330, samples=300] plot ({4+2*cos(\x)}, {1.6*sin(\x)});
			\draw (5.73,0.8) arc[start angle=-129.77399,end angle=-39.18002, radius=1.46280];
			\draw (12,0) arc[start angle=16.8106,end angle=136.40540, radius=2.49779];
			\draw (7.8,-1) arc[start angle=-136.40540,end angle=-16.8106, radius=2.49783];
			\draw (7.8,-1) arc[start angle=39.18002,end angle=129.77399, radius=1.46305];
			\draw [dashed] (7,-0.5) arc[start angle=-43.59587,end angle=43.59587, radius=0.72511];
			\draw  (7,0.5) arc[start angle=136.3876,end angle=223.6124, radius=0.72489];
			
			\fill (12,0) circle (2pt) node[right] {$P_i$};
			\end{tikzpicture}
		}
		
		\Small{Figure I. There is a closed minimizing geodesic.}
	\end{center}
	
	\begin{center}
		
		\scalebox{0.6}{\begin{tikzpicture}
			
			\draw[domain=30:330, samples=300] plot ({4+2*cos(\x)}, {1.6*sin(\x)});
			\draw (5.73,0.8) arc[start angle=-129.77399,end angle=-39.18002, radius=1.46280];
			\draw (11,0) arc[start angle=2.0217,end angle=143.26971, radius=1.77692];
			\draw (7.8,-1) arc[start angle=-143.26971,end angle=-2.0217, radius=1.77704];
			\draw (7.8,-1) arc[start angle=39.18002,end angle=129.77399, radius=1.46305];
			\draw (7,0.5) arc[start angle=122.82726,end angle=159.71268, radius=1.01248];
			\draw (6.6,0) arc[start angle=-159.75049,end angle=-122.93556, radius=1.01392];
			\draw [dashed] (7,-0.5) arc[start angle=-43.59587,end angle=43.59587, radius=0.72511];
			\fill (6.6,0) circle (2pt) node[left] {$P_i$};
			
			\end{tikzpicture}
		}
		
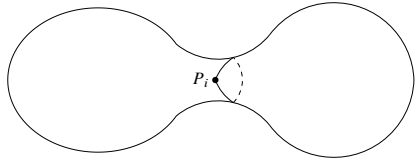
		\Small{Figure II. The singular point is in the neck region, shortest loops may not be smooth.}
	\end{center}

	\subsection{The structure of bubble tree}
	
	Let $p_{1}, \cdots,p_{m}$ be distinct points on $\mathbb{S}^{2}$. 
	From now on, we set
	$g_k=e^{2u_k}g_{\mathbb{S}^2}\in\mathcal{M}(\mathbb{S}^2)$ with
	\begin{align*}
	\K_{g_k}=K_{k}dV_{g_k}-2\pi\sum_{i=1}^m\beta_i^k\delta_{p_i},
	\end{align*}
	in other words, $u_k\in \cap_{p \in[1,2)}W^{1,p}(\mathbb{S}^2,g_{\mathbb{S}^{2}})$  solves the equation
	\begin{align}
	\label{eq.positive on S2}
	-\Delta_{\mathbb{S}^2} u_k=K_ke^{2u_k}-2\pi\sum_{i=1}^{m} \beta_i^k\delta_{p_i}-1
	\end{align}
	in the sense of distributions. 

	For the remainder of this section, we make the following assumptions:
	\begin{itemize}
		\item[(A1)] $K_k$ is continuous and $K_k\rightarrow K$ in $C(\mathbb{S}^2)$, with $K>0$.
		\item[(A2)] $\beta_i>-1$, and $\beta_i^k\rightarrow\beta_i$.
		\item[(A3)] $\chi(\mathbb{S}^2,\beta)=2+\sum\limits_{i=1}^m\beta_i>0$.
	\end{itemize}
	Then, there exist constants $0<a<b$ depending on $K$ and $\beta_{i}$ such that
	$$
	a<\Area(\mathbb{S}^2,g_k)<b.
	$$
	
	Let $c_k$ be the mean value of $u_k$ over $(\mathbb{S}^{2},g_{\mathbb{S}^{2}})$. By Lemma \ref{lemma gradient estimate on Riemann surfaces}, after passing to a subsequence, we assume $u_k-c_k\to v$ weakly in $\cap_{p\in{(1,2)}}W^{1,p}(\mathbb{S}^{2},g_{\mathbb{S}^{2}})$. By Jensen's inequality, $c_{k}$ is bounded from above. If $c_{k}$ is bounded, we may further assume $u_k\to u$ weakly in $\cap_{p\in{(1,2)}}W^{1,p}(\mathbb{S}^{2},g_{\mathbb{S}^{2}})$. 
	We will use the curvature concentration function $\Theta$ to investigate the equations that $u$ and $v$ satisfy. 
	
	\vspace{.1cm}
	\noindent{\bf Note.} From now on, the phrase ``for any $x_{0} \in \mathbb{S}^{2}$, we choose an appropriate isothermal coordinate system with $x_{0}=0$" means: the domain of $u_{k}$ under this coordinate system is $D \subset \mathbb{R}^{2}$, $\{u_{k}\}$ has no bubbles on $D \backslash \{x_{0}\}$, and if $x_{0} \in \{p_{1}, \cdots,p_{m}\}$ then $D \cap \{ p_{1} \cdots,p_{m}\}=\{x_{0}\}$; if $x_{0} \notin \{p_{1}, \cdots,p_{m}\}$ then $D \cap \{ p_{1}, \cdots,p_{m}\}=\emptyset$. 
	
	\vspace{.1cm}		
	We divide the proof of Theorem \ref{theorem metric convergence on S2} into following propositions.

	\begin{pro}
		\label{pro theta's value when no bubble}
		Let $g_{k}=e^{2u_{k}} g_{\mathbb{S}^{2}} \in \mathcal{M}(\mathbb{S}^{2})$ where $u_{k}$ are the solutions of (\eqref{eq.positive on S2}) satisfying \textnormal{(A1), (A2), (A3)}. 
		If $\{u_{k}\}$ has no bubble at $x_{0}$, then
		\begin{align*}
		\Theta(x_{0})=
		\left\{\begin{array}{rll}
		&-2 \pi \beta_i & \text{if } x_{0} =p_{i}\in \{p_{1},\cdots,p_{m}\}, \\
		&0 & \text{if } x_{0} \notin \{p_{1}, \cdots, p_{m}\}. 
		\end{array}\right.
		\end{align*}
	\end{pro}
	
	\begin{proof}
		By choosing an appropriate isothermal coordinate system with $x_{0}=0$, $u_{k}$ solves the following equation locally:
		\begin{align*}
		-\Delta u_{k}=
		\left\{\begin{array}{rll}
		&K_{k}e^{2u_{k}}-2 \pi \beta_i^{k} \delta_{0}  & \text{when } x_{0} \in \{p_{1},\cdots,p_{m}\}, \\
		&K_{k}e^{2u_{k}}  & \text{when } x_{0} \notin \{p_{1}, \cdots, p_{m}\} . 
		\end{array}\right.
		\end{align*}
		Applying Proposition \ref{pro no area concentration at zero on D}, if $x_{0} \notin \{p_{1},\cdots,p_{m}\}$, 
		\begin{align*}
		\Theta(x_{0})&=\lim_{r \rightarrow 0} \lim_{k \rightarrow \infty} \int_{D_{r}} K_{k} e^{2u_{k}}  \leq 2 \|K\|_{C^{0}}    \lim_{r \rightarrow 0} \lim_{k \rightarrow \infty} \int_{D_{r}} e^{2u_{k}} =0;
		\end{align*}
		if $x_{0} =p_{i} \in \{p_{1},\cdots,p_{m}\}$, 
		\begin{align*}
		\Theta(x_{0})&=\lim_{r \rightarrow 0} \lim_{k \rightarrow \infty} \big( \int_{D_{r}} K_{k} e^{2u_{k}} 
		-2 \pi \beta_{i}^{k} \delta_{0} \big)=-2 \pi \beta_{i}. 
		\end{align*}
	\end{proof}
	Denote 
	\begin{align*}
	\mathscr{C}=\{  x \in \mathbb{S}^{2}: \{u_{k} \} \text{ has at least a bubble at  } x  \} \cup \{p_{1}, \cdots,p_{m}\}. 
	\end{align*}
	We know that $	\mathscr{C}$ consists of finite points, then by Proposition \ref{pro theta's value when no bubble}, the following sum is well-defined:
	\begin{align*}
	\sum_{x \in \mathbb{S}^{2} } \Theta(x) \delta_{x}=\sum_{x \in \mathscr{C}} \Theta(x) \delta_{x}.
	\end{align*}
	
	\begin{pro}
		\label{pro equation of u and v using Theta}
		Let $g_{k}=e^{2u_{k}} g_{\mathbb{S}^{2}} \in \mathcal{M}(\mathbb{S}^{2})$ where $u_{k}$ are the solutions of \eqref{eq.positive on S2} satisfying \textnormal{(A1), (A2), (A3)}. Then \\
		\textnormal{(1)} If $u_{k} \rightarrow u$ weakly in $\cap_{p\in{(1,2)}}W^{1,p}(\mathbb{S}^{2},g_{\mathbb{S}^{2}})$, 
		then $u$ solves
		\begin{align*}
		-\Delta_{\mathbb{S}^{2}} u=K e^{2u}+\sum_{x \in \mathbb{S}^{2}} \Theta(x) \delta_{x}-1. 
		\end{align*}
		\textnormal{(2)} If $u_{k}-c_{k} \rightarrow v$ weakly in $\cap_{p\in{(1,2)}}W^{1,p}(\mathbb{S}^{2},g_{\mathbb{S}^{2}})$ with $c_{k} \rightarrow -\infty $, 
		then $v$ solves
		\begin{align*}
		-\Delta_{\mathbb{S}^{2}} v=\sum_{x \in \mathbb{S}^{2}} \Theta(x) \delta_{x}-1.
		\end{align*}
	\end{pro}
	
	\begin{proof}
		Since the proofs of (1) and (2) are almost the same, we only present the proof for (1). 
		
		At any fixed $x_{0} \in \mathbb{S}^{2}$, we choose an appropriate isothermal chart with $x_{0}=0$.  
		Then for any $\varphi\in \mathcal{D}(D)$, 
		\begin{align}
		\int_{D} &\nabla u \nabla\varphi =\lim_{k \rightarrow \infty} \int_{D} \nabla u_{k} \nabla \varphi =\lim_{k \rightarrow\infty}\int_{D} \varphi d\K_{g_k} \nonumber =\lim_{r \rightarrow 0} \lim_{k \rightarrow\infty}\left( \int_{D\setminus D_r}+\int_{D_r}\right) \varphi  d\K_{g_k}\nonumber\\
		&=  \int_D \varphi Ke^{2u}+\varphi(0)\lim_{r \rightarrow 0} \lim_{k \rightarrow\infty}\K_{g_k}(D_r) \nonumber =\int_{D} \varphi K e^{2u}+\Theta(x_{0}) \varphi(0), \nonumber 
		\end{align}
		which yields that locally $u$ solves
		\begin{align*}
		-\Delta u=K e^{2u}+ \Theta(x_{0}) \delta_{0}=K e^{2u}+\sum_{x \in D} \Theta(x) \delta_{X}.  
		\end{align*}
		Therefore, $u$ solves
		\begin{align*}
		-\Delta_{\mathbb{S}^{2}} u=K e^{2u}+\sum_{x \in \mathbb{S}^{2}} \Theta(x) \delta_{x}-1. 
		\end{align*}
	\end{proof}
	
	To obtain the explicit equations that $u$ and $v$ satisfy, we need to compute $\Theta(x_{0})$ when $\{u_{k}\}$ has at least one bubble at $x_{0}$. 
	
	\begin{pro}
		\label{pro Theta's value when uk converges}
		Let $g_{k}=e^{2u_{k}} g_{\mathbb{S}^{2}} \in \mathcal{M}(\mathbb{S}^{2})$ where $u_{k}$ are the solutions of \eqref{eq.positive on S2} satisfying \textnormal{(A1), (A2), (A3)}. 
		Suppose $u_{k} \rightarrow u$ weakly in $\cap_{p\in{(1,2)}}W^{1,p}(\mathbb{S}^{2},g_{\mathbb{S}^{2}})$. If $\{u_{k}\}$ has at least one bubble at some $x_{0} \in \mathbb{S}^{2}$, then
		$x_{0}=p_{i}$ for some $i \in \{1,\cdots,m\}$ satisfying $\beta_{i}>1$, all bubbles of 
		$\{u_{k}\}$ at $x_{0}$ are smooth and at $1$-level. 
		Moreover, if $s$ is the number of bubbles at $x_{0}$, then 
		\begin{align*}
		\Theta(x_{0})=4 \pi s-2 \pi \beta_{i}. 
		\end{align*}
	\end{pro}
	
	\begin{proof}
		If $\{u_{k}\}$ has at least one bubble at some $x_{0} \in \mathbb{S}^{2}$, we choose an appropriate isothermal chart with $x_{0}=0$. Then the following equation holds on $D$:
		\begin{align*}
		-\Delta u_{k}=
		\left\{\begin{array}{rll}
		&K_{k}e^{2u_{k}}-2 \pi \beta_i^{k} \delta_{0}  & \text{when } x_{0} \in \{p_{1},\cdots,p_{m}\}, \\
		&K_{k}e^{2u_{k}}  & \text{when } x_{0} \notin \{p_{1}, \cdots, p_{m}\} . 
		\end{array}\right.
		\end{align*}
		If $x_{0} \notin \{p_{1},\cdots,p_{m}\}$, then 
		$\{u_{k}\}$ has no bubbles at 0 by Lemma \ref{lemma one.level}(3), a contradiction. 
		So $x_{0} \in \{p_{1},\cdots,p_{m}\}$. WLOG, we assume $x_{0}=p_{1}$. Then $u_{k}$ satisfies  
		\begin{align*}
		-\Delta u_{k}=K_{k} e^{2u_{k}}-2 \pi \beta_{1}^{k} \delta_{0}
		\end{align*}
		on $D$. By Lemma \ref{lemma one.level} (3), $\beta_{1}>1$. 
		By Lemma \ref{lemma one.level} (2), all bubbles at 1-level are smooth. 
		We claim that  there is no bubble at 2-level.  Assume $\{ (x_k,r_k) \}$ is a blowup sequence at 1-level, then by Lemma  \ref{lemma one.level} (2), for any $t>0$, $0 \in D_{t} \backslash D_{ \frac{r_{k}}{t} }(x_{k})$ for sufficiently large $k$. Then for any $r>0$, $u_{k}^{\prime}(x)=u_{k}(r_{k}x+x_{k}) +\log r_{k}$ solves the following equation on $D_{r}$ when $k$ is sufficiently large:
		\begin{align*}
		-\Delta u_{k}^{\prime}=K_{k}(r_{k}x+x_{k}) e^{2u_{k}^{\prime}}. 
		\end{align*}
		By Lemma  \ref{lemma one.level} (3), $\{ u_{k}^{\prime} \}$ has no bubble, which means that there does not exist bubbles at 2-level. 
		
		Since each bubble is smooth with total curvature $4\pi$, there exist only finitely many bubbles, saying
		$\{ (x_k^1,r_k^1)\}$, $\cdots$, $\{ (x_k^s,r_k^s) \}$. Then by Theorem \ref{theorem area convergence of bubble tree},
		\begin{align*}
		\Theta(x_{0})&=\lim_{r\rightarrow 0}\lim_{k\rightarrow \infty}\K_{g_k}(D_r )\\
		&= \lim_{r\rightarrow 0}\lim_{k\rightarrow \infty} \Big( \int_{D_{r}} K_{k} e^{2u_{k}}-2 \pi \beta_{1}^{k} \delta_{0} \Big)\\
		&= K(x_{0})\lim_{r\rightarrow 0}\lim_{k\rightarrow \infty} \Area(D_r, g_{k})-2\pi\beta_{1}\\
		&=4\pi s-2\pi \beta_{1}. 
		\end{align*}
	\end{proof}
	
	\begin{center}
		\scalebox{0.7}{
			\begin{tikzpicture}
			\draw[line width=1.3pt] (2.1,1) circle (1pt);
			\draw[line width=0.7pt] (2,1.1) circle (2);
			\draw[line width=0.5pt] (1.3,1.5) circle (0.3);
			\draw[line width=0.5pt] (2.3,0.3) circle (0.4);
			
			\draw (2.4,1) node { $0$};
			\draw (1.3,2.1) node {  $D_{Rr_k^1(x_k^1)}$};
			\draw (2.3,-.3) node {  $D_{Rr_k^1(x_k^2)}$};

			\draw (12,-1) arc(40:140:4);
			\draw[line width=1.3pt] (8.9,0.42) circle (1pt);
			\draw[rotate around={30:(8,2.25)},line width=0.7pt] (8,2.25) ellipse (1 and 2);
			\draw[rotate around={-27:(9.5,1.9)},line width=1pt,opacity=0.3] (9.5,1.9) ellipse (1 and 1.6);

			\draw (8.9,0) node { $0$};
			\draw (7.5,3) node {  $({S}^2,g_{v^1})$};
			\draw (10,2.5) node {  $({S}^2,g_{v^2})$};
			\draw (10,-1) node {  $(D,g_{v})$};
			\end{tikzpicture}}
		
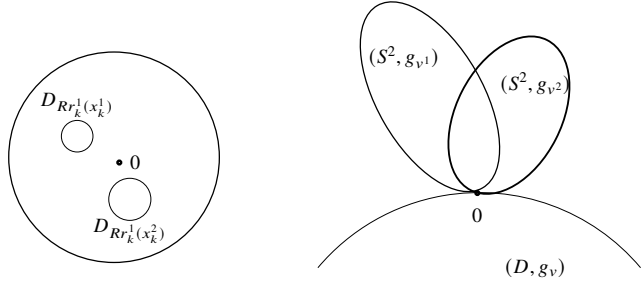
		{\Small{\bf Local figure of Non-Collapsing Case:} \\ All bubbles of $\{u_{k}\}$ at $p_{i}$ with $\beta_{i} >1$ are smooth and at 1-level.}
	\end{center}

	\begin{center}
		\scalebox{0.7}{
			\begin{tikzpicture}[rotate=-270]
			
			\draw[line width=1.3pt]plot[smooth]coordinates{(6,2)(5,0)(6,-1)(8,0)(7,1.2)(6,2)};
			
			\begin{scope}[shift={(-1,-0.3)}]
			\draw[rotate around={19:(6.8,3.4)},line width=0.6pt] (6.8,3.4) ellipse (0.7 and 1.2);
			\draw[rotate around={-27:(7.5,3.1)},line width=1pt,opacity=0.2] (7.5,3.1) ellipse (0.6 and 0.9);
			\end{scope}
			
			\begin{scope}[shift={(4.9, 1.9)}, rotate around={-40:(1,1)},scale=0.5]
			\draw[rotate around={-27:(9.5,1.9)},line width=1pt,opacity=0.4] (9.5,1.9) ellipse (1 and 1.6);
			\end{scope}
			
			\end{tikzpicture}}
		
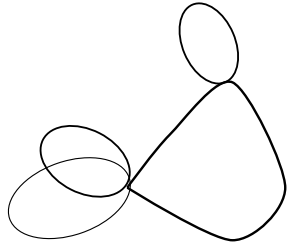
		{\Small{\bf Global figure of Non-collapsing Case:}  \\All bubbles are smooth, and their south poles are attached to a singularity.}
	\end{center}
	
	\begin{pro}
		\label{pro v's equation when uk not converge}
		Let $g_{k}=e^{2u_{k}} g_{\mathbb{S}^{2}} \in \mathcal{M}(\mathbb{S}^{2})$ where $u_{k}$ are solutions of \eqref{eq.positive on S2} satisfying \textnormal{(A1), (A2), (A3)}. Suppose 
		$u_{k} -c_{k}\rightarrow v$ weakly in $\cap_{p\in{(1,2)}}W^{1,p}(\mathbb{S}^{2},g_{\mathbb{S}^{2}})$ with $c_{k} \rightarrow -\infty$. If $\{u_{k}\}$ has at least one bubble at some $x_{0} \in \mathbb{S}^{2}$, then one of the following holds:
		
		\textnormal{(1)} $x_{0} \notin \{p_{1}, \cdots,p_{m}\}$ and 
		$\{u_{k}\}$ has $s$ bubbles  at $1$-level at $x_{0}$. Each bubble is smooth and
		$
		\Theta(x_{0})=4\pi s.
		$
		
		\textnormal{(2)} $x_{0}=$ some $p_{i}$ and  $\{u_{k}\}$ has $s$ bubbles at $1$-level at $x_{0}$. Each bubble is smooth and
		$
		\Theta(x_{0})=4\pi s-2\pi \beta_{i}.
		$
		
		\textnormal{(3)} $x_{0}=$ some $p_{i}$ and $\{u_{k}\}$ has $s$ bubbles at $x_{0}$ while only one of them is singular and the others are all smooth with 
		$
		\Theta(x_{0})=4 \pi s+2 \pi \beta_{i}. 
		$
		
		\textnormal{(4)} $x_{0}=$ some $p_{i}$ and 
		$\{ u_k^i\}$ has s bubbles at $1$-level and $s^{\prime}$ bubbles at $2$-level at $x_{0}$. Exactly one of the bubbles at $1$-level is singular, say $v'$, while the other bubbles at $1$-level are smooth. All bubbles at $2$-level are right on the top of $v'$ and smooth with 
		$
		\Theta(x_{0})=4 \pi (s-s')+2 \pi \beta_{i}.
		$
	\end{pro}
	
	\begin{proof}
		If $\{u_{k}\}$ has at least a bubble at some $x_{0} \in \mathbb{S}^{2}$, we choose an appropriate isothermal coordinate system with $x_{0}=0$. 
		
		Firstly, we show that there are only finitely many bubbles at 1-level.
		We assume $\{ (x_k^1,r_k^1) \}$, $\cdots$, $\{ (x_k^s,r_k^s)\}$ are arbitrary $s$ blowup sequences at 1-level (here we do not say there are only s blowup sequences at 1-level), then for any fixed $R>0$, 
		$$
		D_{Rr_k^i}(x_k^i)\cap D_{Rr_k^j}(x_k^j)=\emptyset,\s i\neq j,
		$$
		when $k$ is sufficiently large. Then we may assume for any $i \in \{2,\cdots,s\}$, 
		for any fixed $R>0$, $0 \notin D_{Rr_k^i}(x_k^i)$ when $k$ is sufficiently large. We set $$
		u_k^i=u_k(x_k^i+r_k^ix)+\log r_k^i,
		$$
		which converges to a bubble $v^i$, then for any $i \in \{2,\cdots,s\}$, for any fixed $R>0$, $u_k^i$
		satisfies the equation
		$$
		-\Delta u_k^i(x)=K_k(x_k^i+r_k^ix)e^{2u_k^i(x)},\s x\in D_R,
		$$
		when $k$ is sufficiently large. Applying Lemma \ref{lemma one.level} (3) to $u_k^i$, $\{ u^i_k \}$ has no bubble,  we conclude that  $u_k^i$
		has no concentration, hence $v^i$ is smooth. Therefore, there is at most one nonsmooth bubble at 1-level, so there are only finitely many bubbles at 1-level.
		
		Now we  may assume that 
		$\{ (x_k^1,r_k^1)\}$, $\cdots$, $\{ (x_k^s,r_k^s)\}$ are exactly all  blowup sequences at 1-level. We divide the argument into the following different cases. 
		
		{\bf Case 1:} $x_{0} \notin \{p_{1},\cdots,p_{m}\}$. For this case, $u_{k}$ solves
		\begin{align*}
		-\Delta u_{k}=K_{k}  e^{2u_{k}}. 
		\end{align*} 
		Then for any $i \in \{1,\cdots,s\}$ and any fixed $R$, $0 \notin D_{Rr_k^i}(x_k^i)$ when $k$ is large. So
		each $v^i$ is smooth and there are no bubbles at 2-level. 
		By Theorem \ref{theorem area convergence of bubble tree}, 
		\begin{align*}
		\Theta(x_{0})&=\lim_{r\rightarrow 0}\lim_{k\rightarrow+\infty}\K_{g_k}(D_r)\\
		&=K(x_{0})\lim_{r\rightarrow 0}\lim_{k\rightarrow+\infty}\Area(D_r)\\
		&=K(x_{0})\sum_{i=1}^{s} \Area(\R^2,g_{v^i})\\
		&= 4 \pi s.
		\end{align*}
		
		{\bf Case 2:} $x_{0} \in \{p_{1},\cdots,p_{m}\}$. 
		WLOG, we assume $x_{0}=p_{1}$. Then
		\begin{align*}
		-\Delta u_{k}=K_{k}  e^{2u_{k}}-2 \pi \beta_{1}^{k} \delta_{0}. 
		\end{align*} 
		
		{\bf Case 2.1:} 
		For any $i \in \{1,\cdots,s\}$ and any fixed $R>0$, $0 \notin D_{Rr_k^i}(x_k^i)$ when $k$ is large. Then
		each $v^i$ is smooth and no bubbles are at 2-level. By Theorem \ref{theorem area convergence of bubble tree}, 
		\begin{align*}
		\Theta(x_{0})&=\lim_{r\rightarrow 0}\lim_{k\rightarrow+\infty}\K_{g_k}(D_r)\\
		&= K(x_{0})\lim_{r\rightarrow 0}\lim_{k\rightarrow+\infty}\Area(D_r)-2\pi \beta_{1}\\
		&= K(x_{0})\sum_{i=1}^{s} \Area(\R^2,g_{v^i})-2\pi \beta_{1}\\
		&=4\pi s-2\pi \beta_{1}.
		\end{align*}
		
		\begin{center}
			\scalebox{0.7}
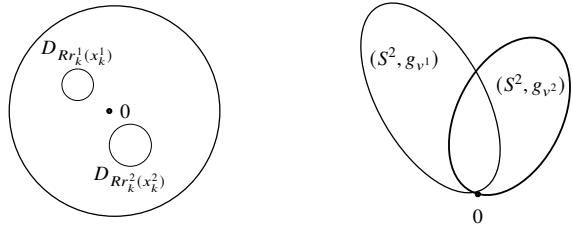
{
				\begin{tikzpicture}
				

				\draw[line width=1.3pt] (1.9,0) circle (1pt);
				\draw[line width=0.7pt] (2,0) circle (2);
				\draw[line width=0.5pt] (1.3,.5) circle (0.3);
				\draw[line width=0.5pt] (2.3,-.65) circle (0.4);
				
				\draw (2.2,0) node { $0$};
				\draw (1.3,1.1) node {  $D_{Rr_k^1(x_k^1)}$};
				\draw (2.3,-1.3) node {  $D_{Rr_k^2(x_k^2)}$};

				\def \a{2};
				\draw[line width=1.3pt] (8.9,0.42-\a) circle (1pt);
				\draw[rotate around={30:(8,2.25-\a)},line width=0.7pt] (8,2.25-\a) ellipse (1 and 2);
				\draw[rotate around={-27:(9.5,1.9-\a)},line width=1pt,opacity=0.3] (9.5,1.9-\a) ellipse (1 and 1.6);

				\draw (8.9,0-\a) node { $0$};
				\draw (7.5,3-\a) node {  $({S}^2,g_{v^1})$};
				\draw (9.9,2.5-\a) node {  $({S}^2,g_{v^2})$};

				\end{tikzpicture}}
			
			{\Small{\bf Collapsing Case 1 and Case 2.1:}  All bubbles are smooth 
				and at 1-level}\\
		\end{center}
		
		{\bf Case 2.2:}
		There exists $i_{0}$ such that $0\in D_{Rr_k^{i_{0}}}(x_k^{i_{0}})$ for fixed $R$. WLOG, we assume $i_{0}=1$. Then for any $i \in \{2,\cdots,s\}$, 
		for any fixed $R>0$, we have $0 \notin D_{Rr_k^i}(x_k^i)$ when $k$ is large. Hence, $v^{i}$ is smooth for any $i \in \{2,\cdots,s\}$. Now we consider the bubble $v^{1}$. Set $y_k^1=\frac{-x_k^1}{r_k^1}$ and assume $y_k^1\rightarrow y_\infty$. Then 
		$$
		-\Delta u_k^1=K_k(x_k^1+r_k^1x)e^{2 u_k^1}-2\pi\beta_{1}^{k} \delta_{y_k}.
		$$
		By arguments similar to those for Proposition \ref{pro equation of u and v using Theta}, there exists $\tau$ such that	
		$$
		-\Delta v^1=K(x_{0})e^{2v^1}+\tau \delta_{y_\infty}.
		$$
		
		We further divide {\bf Case 2.2}  into two cases.
		
		{\bf Case 2.2.1:} $\{u_k^1\}$ has no  bubble, 
		i.e. no bubbles at 2-level. Then
		$$
		-\Delta v^1=K(x_{0})e^{2v^1}-2\pi\beta_{1}\delta_{y_\infty}.
		$$
		By  \cite{CW-L} and \cite{Troyanov1988},  as a metric over $\mathbb{S}^2$, $g_{v^1}$  has exactly 2 singularities and 
		\begin{align*}
		\int_{\mathbb{R}^{2}} K(x_{0}) e^{ 2v^{1}} =4 \pi +4 \pi\beta_{1}. 
		\end{align*}
		Then by Theorem \ref{theorem area convergence of bubble tree},
		\begin{align*}
		\Theta(x_{0})&=\lim_{r\rightarrow 0}\lim_{k\rightarrow+\infty}\K_{g_k}(D_r)\\
		&=
		K(x_{0})\lim_{r\rightarrow 0}\lim_{k\rightarrow+\infty}\Area(D_{r}, g_{k})-2\pi\beta_{1}\\
		&=
		K(x_{0})\sum_{i=1}^{s} \Area(\R^2,g_{v^i})-2\pi\beta_{1}\\
		&=(4 \pi +4 \pi \beta_{1}) +4\pi(s-1)-2\pi\beta_{1}=4\pi s +2 \pi \beta_{1}.
		\end{align*}
		
		\begin{center}
			\scalebox{0.7}
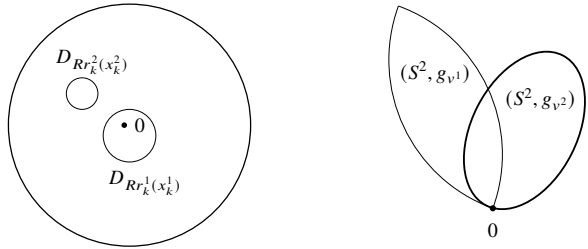
{
				\begin{tikzpicture}
				

				\fill (1.9,0) circle (1.5pt);
				\draw[line width=0.7pt] (2,0) circle (2.3);
				\draw[line width=0.5pt] (1.1,.6) circle (0.3);
				\draw[line width=0.5pt] (2,-.2) circle (0.5);
				
				\draw (2.2,0) node { $0$};
				\draw (1.3,1.2) node {  $D_{Rr_k^2(x_k^2)}$};
				\draw (2.3,-1.1) node {  $D_{Rr_k^1(x_k^1)}$};

				\def \a{2};
				\draw[line width=1.3pt] (8.9,0.42-\a) circle (1pt);
				\draw[rotate around={-27:(9.5,1.9-\a)},line width=1pt,opacity=0.3] (9.5,1.9-\a) ellipse (1 and 1.6);
				\draw[line width=0.5pt] (8.9,0.42-\a) arc(-20:70:3);
				\draw[line width=0.5pt] (8.9,0.42-\a) arc(70:-20:-3);
				
				\draw (8.9,0-\a) node { $0$};
				\draw (7.8,3-\a) node {  $({S}^2,g_{v^1})$};
				\draw (9.8,2.5-\a) node {  $({S}^2,g_{v^2})$};
				
				\end{tikzpicture}}
			
			{\Small{\bf Collapsing Case 2.2.1:}  All bubbles are at 1-level; 
				only one of them is singular}\\
		\end{center}

		{\bf Case 2.2.2:} $\{ u_k^1 \}$ has bubbles.
		By Lemma \ref{lemma one.level} (2), $\{u_k^1\}$ can only have bubbles at 1-level
		(which are the bubbles of $\{u_k\}$ at 2-level) and all bubbles of $\{u_k^1\}$ are smooth. By Lemma \ref{lemma one.level} (3), $\beta_{1} > 1$. 
		Similar to the arguments in Proposition \ref{pro Theta's value when uk converges}, 
		$$
		\tau=4 \pi s'-2 \pi \beta_{1} ,
		$$
		where $s'$ is the number of the bubbles of $\{ u_k^1 \}$.  
		By  \cite{CW-L} and \cite{Troyanov1988} again, 
		\begin{align*}
		\int_{\mathbb{R}^{2}} K(x_{0}) e^{ 2v^{1}} =4 \pi -2 \tau. 
		\end{align*}
		
		Let  $\{ (x_k^{s+1},r_k^{s+1}) \}$, $\cdots$, $\{ (x_k^{s+s'},r_k^{s+s'}) \}$ be all of the blowup sequences right on the top of $\{ (x_k^1,r_k^1)\}$. Then for any $i \in \{ s+1, \cdots, s+s^{\prime} \}$, $\{ (x_k^{i},r_k^{i}) \}$ converges to a smooth bubble and  $\{ (x_k^i,r_k^i)\}$ has no bubbles.
		By Theorem \ref{theorem area convergence of bubble tree}, we have
		\begin{align*}
		\Theta(x_{0})&=\lim_{r\rightarrow 0}\lim_{k\rightarrow+\infty}\K_{g_k}(D_r)\\
		&=
		K(x_{0})\lim_{r\rightarrow 0}\lim_{k\rightarrow+\infty} \Area(D_r,g_k)-2\pi\beta_{1}\\
		&=
		K(x_{0})\sum_{i=1}^{s+s^{\prime}} \Area(\R^2,g_{v^i})-2\pi\beta_{1}\\
		&=4 \pi (s-1)+(4 \pi -2 \tau)+4\pi s^{\prime}-2\pi\beta_{1}\\
		&=4\pi(s-s')+2 \pi \beta_{1}.
		\end{align*}
		
		\begin{center}
			\scalebox{0.7}
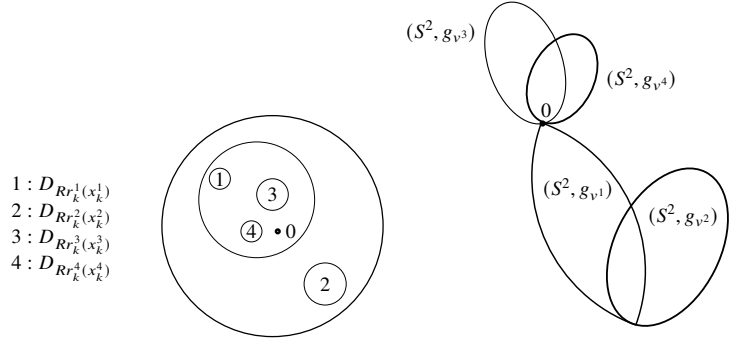
{
				\begin{tikzpicture}
				
				
				\draw[line width=1.3pt] (2.1,0.2) circle (1pt);
				\draw[line width=0.7pt] (2,0.3) circle (2.1);
				\draw[line width=0.5pt] (1.7,0.8) circle (1.1);
				\draw[line width=0.5pt] (3,-0.8) circle (0.4);
				\draw[line width=0.5pt] (2,0.9) circle (0.3);
				\draw[line width=0.5pt] (1.6,0.2) circle (0.2);
				\draw[line width=0.5pt] (1,1.2) circle (0.2);
				
				\draw (2.35,0.2) node { $0$};
				\draw (1,1.2) node {  $1$};
				\draw (3,-0.8) node {  $2$};
				\draw (2,0.9) node {  $3$};
				\draw (1.6,0.2) node {  $4$};
				
				\draw (-2,1) node {  $1: D_{Rr_k^1(x_k^1)}$};
				\draw (-2, 0.5) node {  $2: D_{Rr_k^2(x_k^2)}$};
				\draw (-2,0) node {  $3: D_{Rr_k^3(x_k^3)}$};
				\draw (-2,-.5) node {  $4: D_{Rr_k^4(x_k^4)}$};
				
				\def \a{2};
				\draw[line width=1.3pt] (7.13,4.25-\a) circle (1pt);
				\draw[rotate around={-27:(9.5,1.9-\a)},line width=1pt,opacity=0.4] (9.5,1.9-\a) ellipse (1 and 1.6);
				\draw[line width=0.8pt] (8.9,0.42-\a) arc(-20:70:3);
				\draw[line width=0.8pt] (8.9,0.42-\a) arc(70:-20:-3);

				\draw[rotate around={19:(6.8,3.4)},line width=0.6pt] (6.8,3.4) ellipse (0.7 and 1.2);
				\draw[rotate around={-27:(7.5,3.1)},line width=1pt,opacity=0.2] (7.5,3.1) ellipse (0.6 and 0.9);
				
				\draw (7.2,4.5-\a) node { $0$};
				\draw (7.8,3-\a) node {  $({S}^2,g_{v^1})$};
				\draw (9.8,2.5-\a) node {  $({S}^2,g_{v^2})$};
				\draw (5.2,4) node {  $({S}^2,g_{v^3})$};
				\draw (9,3.1) node {  $({S}^2,g_{v^4})$};
				
				\end{tikzpicture}}

			{\Small{\bf Collapsing Case 2.2.2:} \\ There are two levels;
				the only singular bubble is at 1-level;\\
				all bubbles at 2-level are right on the top of the singular bubble.}\\
		\end{center}
		

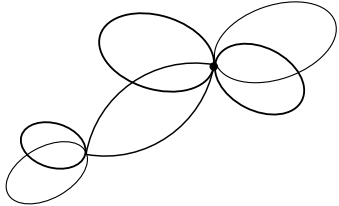
		\begin{center}
			\scalebox{0.7}{
				\begin{tikzpicture}[rotate=-90]
				
				\begin{scope}[shift={(-1,-0.3)}]
				\draw[rotate around={19:(6.8,3.4)},line width=0.6pt] (6.8,3.4) ellipse (0.7 and 1.2);
				\draw[rotate around={-27:(7.5,3.1)},line width=1pt,opacity=0.2] (7.5,3.1) ellipse (0.6 and 0.9);
				\end{scope}
				
				\begin{scope}[shift={(10.5, 1.15)}, rotate around={-170:(1,1)},scale=0.7]
				\draw[rotate around={-27:(9.5,1.9)},line width=1pt,opacity=0.4] (9.5,1.9) ellipse (1 and 1.6);
				\draw[line width=0.8pt] (8.9,0.42) arc(-20:70:3);
				\draw[line width=0.8pt] (8.9,0.42) arc(70:-20:-3);

				\draw[rotate around={19:(6.8,5.4)},line width=0.6pt] (6.8,5.4) ellipse (0.7 and 1.2);
				\draw[rotate around={-27:(7.5,5.1)},line width=1pt,opacity=0.2] (7.5,5.1) ellipse (0.6 and 0.9);
				\end{scope}
				
				\fill (6.26,1.93) circle (2.3pt);
				\end{tikzpicture}}
			
			{\Small{\bf Global figure of Collapsing Case:} \\
				There are at most two levels.  The first level consists of smooth spheres and a sphere with two singular points. The South Poles of the smooth spheres  and one of the singular points of the nonsmooth sphere are attached to a point (the limit of  $g_{k}$). The second level consists of smooth spheres with their South Poles attached to the other singular point of a nonsmooth sphere.}
		\end{center}
	\end{proof}
	
	Theorem \ref{theorem metric convergence on S2} now follows from  combining Proposition \ref{pro theta's value when no bubble}, Proposition \ref{pro equation of u and v using Theta}, Proposition \ref{pro Theta's value when uk converges} and Proposition \ref{pro v's equation when uk not converge}.

	Under the assumptions of Proposition \ref{pro v's equation when uk not converge}, if we further assume $\{u_{k}\}$ has only bubbles at 1-level, we will obtain a numerical relation between the number of the bubbles and a linear combination of $\beta$'s components. Furthermore, both the number of singular bubbles and that of smooth bubbles can be controlled by $\chi(\mathbb{S}^{2},\beta)$. 
	
	\begin{pro}
		\label{pro uk not converge and one level}
		Let $g_{k}=e^{2u_{k}} g_{\mathbb{S}^{2}} \in \mathcal{M}(\mathbb{S}^{2})$ where $u_{k}$ are solutions of \eqref{eq.positive on S2} satisfying \textnormal{(A1), (A2), (A3)}. We assume that $u_{k} -c_{k}\rightarrow v$ weakly in $\cap_{p\in{(1,2)}}W^{1,p}(\mathbb{S}^{2},g_{\mathbb{S}^{2}})$ with $c_{k} \rightarrow -\infty$ and $\{u_{k}\}$ only have bubbles at 1-level. If $\{u_{k}\}$ has singular bubbles at $p_{1}, \cdots, p_{j_{0}}$ and $\{u_{k}\}$ has $t$ smooth bubbles, then $\{u_{k}\}$ has 
		$$
		s=t+j_{0}= \frac{1}{2} \chi(\mathbb{S}^{2}, \beta)-\sum_{i=1}^{j_{0}} \beta_{i}
		$$
		bubbles. Moreover, 
		\begin{align*}
		\max_{1\leq i\leq j_{0}} \beta_i \leq \frac{1}{2} \chi(\mathbb{S}^{2},\beta)-1, \quad t \leq \frac{1}{2} \chi(\mathbb{S}^{2},\beta). 
		\end{align*}
	\end{pro}	
	
	\begin{proof}
		By our assumptions, \textbf{Case 2.2.2} in the proof of Proposition \ref{pro v's equation when uk not converge} will not occur. Then we may assume that 
		$\{ u_k \}$ has $t_i$ smooth bubbles and one singular bubble at $p_{i}$ for $i=1,\cdots,j_{0}$, $t_{i}$ smooth bubbles at $p_{i}$ for $i=j_{0}+1,\cdots,j_{1}$  and $t_{i}^{\prime}$ smooth bubbles at $q_{i}$ for $i=1, \cdots, j_{2}$ where $q_{i} \notin \{p_{1}, \cdots,p_{m}\}$. Then $\{u_{k}\}$ has 
		\begin{align*}
		s=\sum_{i=1}^{ j_{0} }(t_{i}+1   )+\sum_{i=j_{0}+1}^{j_{1}} t_{i}+\sum_{i=1}^{j_{2}} t_{i}^{\prime}=j_{0}+t
		\end{align*}
		bubbles. 
		By Proposition \ref{pro equation of u and v using Theta} (2)  and Proposition \ref{pro v's equation when uk not converge}, 
		\begin{align*}
		4 \pi &=\sum_{i=1}^{j_{0}} \big( 4\pi +2 \pi \beta_{i}+4 \pi t_{i} \big)+\sum_{i=j_{0}+1}^{j_{1}} \big(4 \pi t_{i} -2 \pi\beta_{i} \big)+\sum_{i=j_{1}+1}^{m}( -2\pi \beta_{i}  )+\sum_{i=1}^{j_{2}} 4 \pi t_{i}^{\prime} \\
		&=\sum_{i=1}^{j_{0}} (4 \pi +4 \pi \beta_{i})+4 \pi t -2 \pi \sum_{i=1}^{m} \beta_{i},
		\end{align*}
		which yields that
		\begin{align*}
		2 \pi \chi(\mathbb{S}^{2},\beta)=4 \pi + 2 \pi \sum_{i=1}^{m} \beta_{i} =4 \pi(j_{0}+t)+4 \pi \sum_{i=1}^{j_{0}} \beta_{i}. 
		\end{align*}
		Since $\beta_{i}>-1$, 
		\begin{align*}
		4 \pi + 2 \pi \sum_{i=1}^{m} \beta_{i} &= 4 \pi \sum_{i=1}^{j_{0}} (1+\beta_{i} )+4 \pi t
		\\
		&\geq \max\{  4 \pi + 4 \pi \beta_{1}, \cdots, 4 \pi + 4 \pi \beta_{j_{0}}, 4 \pi t \}, 
		\end{align*}
		which yields that 
		\begin{align*}
		\max_{1\leq i\leq j_{0}} \beta_i \leq \frac{1}{2} \sum_{i=1}^{m} \beta_{i}=\frac{1}{2}\chi(\mathbb{S}^{2},\beta)-1, \quad t \leq 1+\frac{1}{2} \sum_{i=1}^{m} \beta_{i}=\frac{1}{2}\chi(\mathbb{S}^{2},\beta). 
		\end{align*}
	\end{proof}
	
	With additional assumptions on $\beta$, the assumptions of Proposition \ref{pro uk not converge and one level} are satisfied and the following corollaries are obtained.

	\begin{cor}
		\label{cor equation of v if chi<2}
		Let $g_{k}=e^{2u_{k}} g_{\mathbb{S}^{2}} \in \mathcal{M}(\mathbb{S}^{2})$ where $u_{k}$ are the solutions of \eqref{eq.positive on S2} satisfying \textnormal{(A1), (A2), (A3)}. 
		Suppose that $\chi(\mathbb{S}^{2},\beta)=2+\sum_{i=1}^{m} \beta_{i} \in (0,2)$ and $\{u_{k}\}$ has at least one bubble. Then \\
		\textnormal{(1)} $u_{k} -c_{k}\rightarrow v$ weakly in $\cap_{p\in{(1,2)}}W^{1,p}(\mathbb{S}^{2},g_{\mathbb{S}^{2}})$ with $c_{k} \rightarrow -\infty$. \\
		\textnormal{(2)}  There exist $s<m$ and 
		$1 \leq i_{1}<i_{2}<\cdots<i_{s} \leq m$ with $\beta_{i_{j}}<0$ such that
		$\{u_{k}\}$ has exactly one singular bubble at $p_{i_{j}}$. Moreover, 
		\begin{align*}
		4 \pi s+4 \pi \sum_{j=1}^{s} \beta_{i_{j}}=2 \pi \sum_{i=1}^{m} \beta_{i}+4 \pi .
		\end{align*}
	\end{cor}
	\begin{proof}
		Since $\chi(\mathbb{S}^{2},\beta) \in (0,2)$, then
		\begin{align*}
		\lim_{k \rightarrow \infty} \int_{\mathbb{S}^{2}} K_{k} dV_{g_{k}}=\lim_{ k \rightarrow \infty} (4 \pi +2 \pi \sum_{i=1}^{m} \beta_{i}^{k})=4 \pi +2 \pi \sum_{i=1}^{m} \beta_{i}<4 \pi, 
		\end{align*}
		which implies that $\{u_{k}\}$ cannot have smooth bubbles. Then by Proposition \ref{pro Theta's value when uk converges}, $u_{k} -c_{k}\rightarrow v$ weakly in $\cap_{p\in{(1,2)}}W^{1,p}(\mathbb{S}^{2},g_{\mathbb{S}^{2}})$ with $c_{k} \rightarrow -\infty$, and by Proposition \ref{pro v's equation when uk not converge}, \textbf{Case 1}, \textbf{Case 2.1} and \textbf{Case 2.2.2} in the proof of Proposition \ref{pro v's equation when uk not converge} cannot happen. Therefore, only  \textbf{Case 2.2.1} can happen with $s=1$ and $\beta_{i}<0$, which means that $\{u_{k}\}$ has exactly one singular bubble at 1-level at some $p_{i}$ with $\beta_{i}<0$. By Proposition \ref{pro equation of u and v using Theta} (2), there exist $s \leq m$ and 
		$1 \leq i_{1}<i_{2}<\cdots<i_{s} \leq m$ with $\beta_{i_{j}}<0$ such that
		\begin{align*}
		-\Delta_{\mathbb{S}^{2}} v=\sum_{j=1}^{s} ( 4 \pi+4 \pi \beta_{i_{j}} ) \delta_{p_{i_{j}}}-2 \pi \sum_{i=1}^{m} \beta_{i} \delta_{p_{i}}-1,
		\end{align*}
		which yields 
		\begin{align*}
		4 \pi s+4 \pi \sum_{j=1}^{s} \beta_{i_{j}}=2 \pi \sum_{i=1}^{m} \beta_{i}+4 \pi .
		\end{align*}
		What left is to show $s<m$. If $s=m$, 
		\begin{align*}
		4 m \pi +4 \pi \sum_{i=1}^{m} \beta_{i}=2 \pi \sum_{i=1}^{m} \beta_{i}+4 \pi,
		\end{align*}
		in turn 
		\begin{align*}
		2-2m=\sum_{i=1}^{m} \beta_{i} \in (-2,0).
		\end{align*}
		This is impossible as $m$ is an integer. 
	\end{proof}
	
	\begin{cor}
		\label{cor equation of v if betai leq 1}
		Let $g_{k}=e^{2u_{k}} g_{\mathbb{S}^{2}} \in \mathcal{M}(\mathbb{S}^{2})$ where $u_{k}$ are the solution of \eqref{eq.positive on S2} satisfying \textnormal{(A1), (A2), (A3)}. Assume $\beta_{i} \leq 1$ for any $i \in \{1,\cdots,m\} $ and $\{u_{k}\}$ has at least one bubble. Then \\
		\textnormal{(1)} $u_{k} -c_{k}\rightarrow v$ weakly in $\cap_{p\in{(1,2)}}W^{1,p}(\mathbb{S}^{2},g_{\mathbb{S}^{2}})$ with $c_{k} \rightarrow -\infty$. \\
		\textnormal{(2)} 
		All bubbles of $\{ u_{k}\}$ are at $1$-level. Further, there exists a set $I \subset  \{1,\cdots,m\}$ such that
		\begin{align*}
		4 \pi s+4 \pi \sum_{i \in I} \beta_{i }=2 \pi \sum_{i=1}^{m} \beta_{i}+4 \pi, 
		\end{align*}
		where $s$ is the number of bubbles of $\{u_{k}\}$. 
	\end{cor}
	\begin{proof}
		Since $\{u_{k}\}$ has at least one bubble and $\beta_{i} \leq 1$, then by Proposition \ref{pro Theta's value when uk converges} $u_{k} -c_{k}\rightarrow v$ weakly in $\cap_{p\in{(1,2)}}W^{1,p}(\mathbb{S}^{2},g_{\mathbb{S}^{2}})$ with $c_{k} \rightarrow -\infty$. By Proposition \ref{pro v's equation when uk not converge}, \textbf{Case 2.2.2} (in the proof) cannot happen, 
		hence the assertion follows from Proposition \ref{pro uk not converge and one level} immediately. \end{proof}

	\subsection{Simplification of bubble-tree convergence when $K_k\to K$ in $C^1$} In this subsection, we will use the Pohozaev inequality to show that if $\{ u_{k} \}$ has at least one bubble and $K_k$ converges in $C^1$ then $c_k\rightarrow-\infty$
	and $u_k$ has only bubbles at 1-level.

	\begin{pro}[Pohozaev identity on an annulus]
		\label{prop Pohozaev identity}
		Assume $K \in C^{1}(D\backslash \overline{D_\delta})$ and 
		\begin{align*}
		-\Delta u = Ke^{2u} \text{ on } D\backslash \overline{D_\delta}.
		\end{align*}
		Define a function
		$$
		P(t) = t\int_{\partial D_{t}}\Big(\Big|\frac{\partial u}{\partial r}\Big|^{2}-\frac{1}{r^{2}}\Big|\frac{\partial u}{\partial \theta}\Big|^{2}\Big)+2 \int_{ \partial D_{t} } \frac{\partial u}{\partial r}+ t \int_{\partial D_t} Ke^{ 2u }.
		$$
		Then for any $\delta<s<t<1$, we have
		$$
		P(t) - P(s) =  \int_{D_{t} \backslash D_{s}} re^{2u} \frac{\partial K}{\partial r}.
		$$
	\end{pro}
	
	\begin{lem}
		\label{lemma one level bubble K C1}
		Under the assumptions of Lemma \ref{lemma one.level}, if we further assume $K_{k}$ converges to a positive function $K$ in $C^{1}(\overline{D})$, then $\{ u_{k} \}$ has no bubbles. 
	\end{lem}
	
	\begin{proof}
		We argue by contradiction. Assume $\{u_{k}\}$ has at least one bubble. Then by Lemma \ref{lemma one.level}, there exist $s\geq 1$ smooth bubbles at 1-level at $y$. 
		By Theorem \ref{theorem area convergence of bubble tree}, 
		\begin{align*}
		\lim_{r \rightarrow 0} \lim_{k \rightarrow +\infty} \int_{ D_{r}(y)} K_{k} e^{2u_{k}}=4 \pi s, 
		\end{align*}
		which yields 
		\begin{align*}
		-\Delta u= K e^{2u}-2\pi (\lambda-2s) \delta_{y}.
		\end{align*}
		For convenience, we set $\lambda^{\prime}=\lambda-2s$. Define
		\begin{align*}
		P_k(t)&=t\int_{\partial D_{t}(y_k)}\Big(\Big|\frac{\partial u_k}{\partial r}\Big|^{2}-\frac{1}{r^{2}}\Big|\frac{\partial u_k}{\partial \theta}\Big|^{2}\Big)+2 \int_{ \partial D_{t}(y_k) } \frac{\partial u_k}{\partial r}+ t \int_{\partial D_t(y_k)} K_{k}e^{ 2u_k } ,\\
		P(t)&=t\int_{\partial D_{t}(y)}\Big(\Big|\frac{\partial u}{\partial r}\Big|^{2}-\frac{1}{r^{2}}\Big|\frac{\partial u}{\partial \theta}\Big|^{2}\Big)+2 \int_{ \partial D_{t}(y) } \frac{\partial u}{\partial r}+ t \int_{\partial D_t(y)} K e^{ 2u} .
		\end{align*}		
		By Proposition \ref{prop Pohozaev identity}, for any $0<s<t<\frac{1}{2}$ and sufficiently large $k$, 
		\begin{align*}
		|P_{k}(t)-P_{k}(s) | =\Big|\int_{D_{t}(y_{k}) \setminus D_{s}(y_{k}) } re^{2u} \frac{\partial K_{k} }{\partial r} \Big| \leq t \|K_{k } \|_{C^{1}( \overline{D} )} \Area(D,g_{k}),
		\end{align*}
		which yields 
		$$
		\lim_{t\rightarrow 0}\lim_{k\rightarrow +\infty}\lim_{s\rightarrow 0}|P_k(t)-P_k(s)|=0.
		$$
		Let's calculate $\lim_{k\rightarrow +\infty} \lim_{s\rightarrow 0}P_k(s)$ and $\lim_{t\rightarrow 0}\lim_{k\rightarrow +\infty}P_k(t)$ step by step.
		
		Since
		$$
		-\Delta u_{k} =K_{k}e^{2u_{k}}-2 \pi \lambda_k \delta_{y_k},
		$$
		we may write $u_{k}=v_{k}+\lambda_{k} \log r$, where 
		$$
		-\Delta v_{k}=K_{k} e^{2u_{k}}. 
		$$	
		Since $\lambda_k>-1$, for large $j$ we have $K_ke^{2u_k}\in L^p(D_{ 2^{-j}}(y_{k}))$ for some $p>1$. Then $v_{k} \in W^{2,p}(D_{ 2^{-j}  }(y_{k})) \subset  W^{1,2}(D_{ 2^{-j}  }(y_{k})) \cap C^{0}(D_{ 2^{-j}  }(y_{k}))$. It follows 
		\begin{align*}
		\int_{ D_{2^{-j}} (y_{k})  \setminus D_{2^{-j-1}}(y_{k})    } \big|\nabla v_{k}\big|^{2} \rightarrow 0,\quad \text{as } j\rightarrow +\infty.
		\end{align*}
		Hence, there exists $s_{j} \in (  2^{-j-1},2^{-j})$ such that
		\begin{align*}
		s_{j} \int_{\partial D_{s_{j}}(y_{k})} \big|\nabla v_{k}\big|^{2} \rightarrow 0,\quad \text{as } j \rightarrow +\infty.
		\end{align*}
		By direct calculations, 
		\begin{align*}
		&\lim_{j \rightarrow\infty} s_{j} \int_{\partial D_{s_{j}}}  \Big|\frac{\partial u_{k}}{\partial r} \Big|^{2}=\lim_{j  \rightarrow\infty} s_{j} \int_{\partial D_{s_{j}}} \Big|\frac{\partial v_{k}}{\partial r}+\frac{\lambda_{k}}{r}\Big|^{2}=2 \pi \lambda_k^{2}, \\
		&\lim_{j \rightarrow \infty} s_{j} \int_{\partial D_{s_{j}}} \frac{1}{r^{2}} \Big|\frac{\partial u_k}{\partial \theta} \Big|^{2}  \leq \lim_{j \rightarrow \infty} s_{j} \int_{\partial D_{s_{j}}} \big|\nabla v_{k}\big|^{2}=0,\\
		&\lim_{j  \rightarrow \infty} \int_{\partial D_{s_{j}}(y_{k})} \frac{\partial u_{k}}{\partial r}=\lim_{j  \rightarrow \infty} \Big( 2 \pi \lambda_k+\int_{\partial D_{s_{j}}(y_{k})} \frac{\partial v_{k}}{\partial r}\Big)\\&\hspace{1.2in} =2 \pi \lambda_k-\lim_{j  \rightarrow \infty}\int_{D_{s_{j}} (y_{k})} K_{k}e^{2u_{k}} =2 \pi \lambda_k,\\
		& \lim_{j  \rightarrow \infty}  s_{j} \int_{ \partial D_{s_{j}}(y_{k}) } K_{k} e^{2u_{k}}= 
		\lim_{j  \rightarrow \infty} s_{j}^{2} \int_{0}^{2 \pi} K_{k}(s_{j},\theta)   e^{2u_{k}(s_{j},\theta)}\\
		&\hspace{1.5in} \leq C\lim_{j \rightarrow \infty} s_{j}^{2+2\lambda_k}=0.
		\end{align*}
		Thus, we obtain
		$$
		\lim_{k\rightarrow\infty}\lim_{j \rightarrow\infty}P_{k}(s_j)=2\pi\lambda^2+4\pi\lambda.
		$$
		Since $u_k$ converges to $u$ in $C^{2,\alpha}_{\loc}(D\backslash\{y\})$, then by similar calculations, 
		$$
		\lim_{t\rightarrow 0}\lim_{k\rightarrow\infty}P_{k}(t)=\lim_{t\rightarrow 0}P(t)
		=2\pi{\lambda^{\prime}}^2+4\pi\lambda^{\prime}.
		$$
		Then 
		$$
		(\lambda-\lambda^{\prime})(\lambda+\lambda^{\prime}+2)=0.
		$$
		Since $\lambda^{\prime}=\lambda-2s<\lambda$, we obtain
		$$
		\lambda^{\prime}=-\lambda -2<-1
		$$
		which leads to a contradiction to Lemma \ref{lemma removable singularity}. 
	\end{proof}

	\begin{pro}
		\label{pro v's equation when Kk converges in C1}
		Let $g_{k}=e^{2u_{k}} g_{\mathbb{S}^{2}} \in \mathcal{M}(\mathbb{S}^{2})$ where $u_{k}$ are the solutions of \eqref{eq.positive on S2}. Assume \textnormal{(A2)}, \textnormal{(A3)} hold and $K_{k} \rightarrow K$ in $C^{1}(\mathbb{S}^{2})$ with $K>0$.  If $\{u_{k}\}$ has at least one bubble, then\\
		\textnormal{(a)} 
		$u_{k} -c_{k}\rightarrow v$ weakly in $\cap_{p\in{(1,2)}}W^{1,p}(\mathbb{S}^{2},g_{\mathbb{S}^{2}})$ with $c_{k} \rightarrow -\infty$. \\
		\textnormal{(b)} All bubbles of $\{u_{k}\}$ are at $1$-level. More precisely, if $\{u_{k}\}$ has at least one bubble at some $x_{0} \in \mathbb{S}^{2}$, then one of the following holds:\\
		\begin{enumerate}
			\item $x_{0}=$ some $p_{i}$, $\{u_{k}\}$ has $s=\beta_{i}+1$ smooth bubbles at $1$-level at $x_{0}$,
			and
			$$
			\Theta(x_{0})=4 \pi  +2 \pi \beta_{i}. 
			$$
			\item $x_{0}=$ some $p_{i}$, $\{u_{k}\}$ has one singular bubble at $x_{0}$ at 1-level, and 
			$$
			\Theta(x_{0})=4\pi +2\pi \beta_{i}. 
			$$
			\item$x_{0} \notin \{p_{1}, \cdots,p_{m}\}$, 
			$\{u_{k}\}$ has one smooth bubble  at $1$-level at $x_{0}$, and
			$$
			\Theta(x_{0})=4\pi. 
			$$
		\end{enumerate}
		\textnormal{(c)} Furthermore, there exists  a set $I \subset  \{1,\cdots,m\}$ such that
		\begin{align*}
		\#\,\{ \textnormal{ bubbles of $\{u_k\}$}\} =	\frac{1}{2} \chi(\mathbb{S}^{2},\beta)-\sum_{i \in I} \beta_{i},\leq 1+\frac{1}{2} \sum_{i=1}^{m} |\beta_{i}|.
		\end{align*}
		
	\end{pro}
	
	\begin{proof}
		(a) follows from Lemma \ref{lemma one level bubble K C1} immediately. 
		\vspace{.1cm}
		
		\textbf{Step 1:} We prove (b) by investigating the structure of bubble trees. 
		
		For any fixed point $x_{0} \in \mathbb{S}^{2}$, if $\{u_{k}\}$ has at least one bubble at $x_{0}$, we choose an appropriate isothermal coordinate system with $x_{0}=0$. Via similar arguments as in the proof of Proposition \ref{pro v's equation when uk not converge}, we may assume 
		$\{u_{k}\}$ has $s$ blowup sequences $\{( x_{k}^{1}, r_{k}^{1}  )\}, \cdots, \{  (x_{k}^{s} , r_{k}^{s}  ) \}$ at $x_0$ at 1-level, and let $v^{1}, \cdots, v^{s}$ the corresponding bubbles. Set
		\begin{align*}
		u_{k}^{i}(x)=u_{k}(  x_{k}^{i} +r_{k}^{i}x    )+\log r_{k}^{i}.
		\end{align*}
		Then for any fixed $R>0$, when $k$ is sufficiently large, 
		\begin{align*}
		D_{R r_{k}^{i}}(x_{k}^{i } ) \cap D_{R r_{k}^{j}}(x_{k}^{j} ) =\emptyset, \quad i \neq j. 
		\end{align*}
		By Lemma \ref{lemma one level bubble K C1},  $\{ u_{k}^{i}(x)\}$ has no bubbles. 
		
		\noindent \textbf{Case 1:} $x_{0} \in \{p_{1}, \cdots, p_{m}  \}$. 
		WLOG, we may assume $x_{0}=p_{1}$, then $u_{k}$ solves 
		\begin{align*}
		-\Delta u_{k}=K_{k} e^{2u_{k}}-2 \pi \beta_{1}^{k} \delta_{0}. 
		\end{align*}
		
		\noindent \textbf{Case 1.1:} For any $j \in \{1,\cdots,s\}$ and any fixed $R>0$, $0 \notin D_{Rr_k^i}(x_k^i)$ for large $k$. Then
		each $v^i$ is smooth and no bubbles are at 2-level. By Theorem \ref{theorem area convergence of bubble tree}, we have
		\begin{align*}
		\Theta(x_{0})&=\lim_{r\rightarrow 0}\lim_{k\rightarrow+\infty}\K_{g_k}(D_r)\\
		&=K(x_{0})\lim_{r\rightarrow 0}\lim_{k\rightarrow+\infty}\Area(D_r)-2\pi \beta_{1}\\
		&=K(x_{0})\sum_{i=1}^{s} \Area(\R^2,g_{v^i})-2\pi \beta_{1}\\
		&=4\pi s-2\pi \beta_{1}.
		\end{align*}
		
		\noindent \textbf{Case 1.2:}  
		There exists $i_{0}$ such that $0\in D_{Rr_k^{i_{0}}}(x_k^{i_{0}})$ for fixed $R$. 
		WOLG, we assume $i_{0}=1$, then for any $i \in \{2,\cdots,s\}$, 
		for any fixed $R>0$, $0 \notin D_{Rr_k^i}(x_k^i)$ when $k$ is sufficiently large. Hence for any $i \in \{2,\cdots,s\}$, $v^{i}$ is smooth. 
		We set $y_k^1=\frac{-x_k^1}{r_k^1}$ and assume $y_k^1\rightarrow y_\infty$. Then $u_k^1$ satisfies the equation
		$$
		-\Delta u_k^1=K_k(x_k^1+r_k^1x)e^{2 u_k^1}-2\pi\beta_{1}^{k} \delta_{y_k}.
		$$
		Since $\{  u_{k}^{1} \}$ has no bubbles, then
		\begin{align*}
		-\Delta v^{1}=K(x_{0}) e^{ 2v^{1}  }-2 \pi \beta_{1} \delta_{Y}, \quad \int_{\mathbb{R}^{2}} K(x_{0}) e^{2v^{1}}=4 \pi +4 \pi \beta_{1}. 
		\end{align*}
		Then by Theorem \ref{theorem area convergence of bubble tree},
		\begin{align*}
		\Theta(x_{0})&=\lim_{r\rightarrow 0}\lim_{k\rightarrow+\infty}\K_{g_k}(D_r)\\
		&=K(x_{0})\lim_{r\rightarrow 0}\lim_{k\rightarrow+\infty}\Area(D_{r}, g_{k})-2\pi\beta_{1}\\
		&=K(x_{0})\sum_{i=1}^{s} \Area(\R^2,g_{v^i})-2\pi\beta_{1}\\
		&=(4 \pi +4 \pi \beta_{1}) +4\pi(s-1)-2\pi\beta_{1}=4\pi s +2 \pi \beta_{1}.
		\end{align*}
		
		\noindent \textbf{Case 2:} $x_0\notin\{p_1,\cdots,p_m\}$. 
		For this case, $u_{k}$ solves
		\begin{align*}
		-\Delta u_{k}=K_{k}  e^{2u_{k}}. 
		\end{align*} 
		Then for any $i \in \{1,\cdots,s\}$ and any fixed $R>0$, $0 \notin D_{Rr_k^i}(x_k^i)$ when $k$ is  large. So 
		each $v^i$ is smooth and no bubbles are at 2-level. 
		By Theorem \ref{theorem area convergence of bubble tree}, we have
		\begin{align*}
		\Theta(x_{0})&=\lim_{r\rightarrow 0}\lim_{k\rightarrow+\infty}\K_{g_k}(D_r)\\
		&= K(x_{0})\lim_{r\rightarrow 0}\lim_{k\rightarrow+\infty}Area(D_r)\\
		&=K(x_{0})\sum_{i=1}^{s} \Area(\R^2,g_{v^i})\\
		&=4\pi s.
		\end{align*}
		Therefore, $v$ solves the following equation locally:
		\begin{align*}
		-\Delta v=\Theta(x_{0}) \delta_{0}, 
		\end{align*}
		where $\Theta(x_{0}) =4 \pi s-2 \pi \beta_{1}$ (when \textbf{Case 1.1} holds) or 
		$4 \pi s+2 \pi \beta_{1}$ (when \textbf{Case 1.2} holds) or $4 \pi s$ (when \textbf{Case 2} holds). 
		
		Now we calculate $\Theta(x_{0}) $ more precisely via Proposition \ref{prop Pohozaev identity}. Define 
		$$
		P_k(t)=t\int_{\partial D_{t}}\Big(\Big|\frac{\partial u_k}{\partial r}\Big|^{2}-\frac{1}{r^{2}}\Big|\frac{\partial u_k}{\partial \theta}\Big|^{2}\Big)+2 \int_{ \partial D_{t} } \frac{\partial u_k}{\partial r}+ t \int_{\partial D_t} Ke^{ 2u_k }. 
		$$
		By Proposition \ref{prop Pohozaev identity}, we have
		$$
		\lim_{t\rightarrow 0}\lim_{k\rightarrow +\infty}\lim_{s\rightarrow 0}|P_k(t)-P_k(s)|\leq\lim_{t\rightarrow 0}\lim_{k\rightarrow +\infty}\lim_{s\rightarrow 0}Ct\Area(\mathbb{S}^{2}, g_k)=0.
		$$
		Since
		\begin{align*}
		-\Delta u_{k}=
		\left\{\begin{array}{rll}
		&K_{k}e^{2u_{k}}-2 \pi \beta_1^{k} \delta_{0}  & \text{when } x_{0}=p_{1}, \\
		&K_{k}e^{2u_{k}}  & \text{when } x_{0} \notin \{p_{1}, \cdots, p_{m}\} ,
		\end{array}\right.
		\end{align*}
		then as in the proof of Lemma \ref{lemma one level bubble K C1},
		there exists $s_{j} \rightarrow 0$ such that
		\begin{align*}
		\lim_{k\rightarrow+\infty}\lim_{j \rightarrow+\infty}P_{k}(s_j)=
		\left\{\begin{array}{rll}
		&2\pi(\beta_1)^2+4\pi\beta_1  & \text{when } x_{0}=p_{1}, \\
		&0 & \text{when } x_{0} \notin \{p_{1}, \cdots, p_{m}\} .
		\end{array}\right.
		\end{align*}
		On the other hand, since $c_{k } \rightarrow -\infty$ and  $v=-\frac{\Theta(x_{0}) }{2 \pi } \log r+V$, where $V$ is harmonic on $D$, then
		\begin{align*}
		&\lim_{t \rightarrow 0} \lim_{k\rightarrow+\infty}\int_{\partial D_{t}} \frac{\partial u_{k}}{\partial r}=\lim_{t \rightarrow 0}  \int_{\partial D_{t}} \frac{\partial v}{\partial r}=-\Theta(x_{0})\\
		&\lim_{t \rightarrow 0} \lim_{k\rightarrow+\infty} t \int_{\partial D_t} K_{k} e^{ 2u_k } \leq \lim_{t \rightarrow 0}   \lim_{k\rightarrow+\infty} Ct \int_{\partial D_{t}} e^{2v} e^{2c_k}=0, \\
		&\lim_{t \rightarrow 0}  \lim_{k  \rightarrow +\infty} t \int_{\partial D_{t}} \Big|\frac{\partial u_{k}}{\partial r}\Big|^{2}=\lim_{t \rightarrow 0}  t \int_{\partial D_{t}}  \Big|\frac{\partial v}{\partial r}\Big|^2=\frac{\Theta(x_{0})^{2}  }{2 \pi }, \\
		&\lim_{t \rightarrow 0}  \lim_{k \rightarrow +\infty} t \int_{\partial D_{t}} \frac{1}{r^{2}} \Big|\frac{\partial u_k}{\partial \theta} \Big|^{2} =\lim_{t \rightarrow 0}  t  \int_{\partial D_{t}} \frac{1}{r^{2}} \Big|\frac{\partial v}{\partial \theta} \Big|^{2}=0,
		\end{align*}
		which yields 
		\begin{align*}
		\lim_{t \rightarrow 0} \lim_{k \rightarrow +\infty} P_{k}(t)=\frac{\Theta(x_{0})^{2}}{2 \pi }-2 \Theta(x_{0})=2 \pi (\lambda^{\prime})^{2}-4 \pi \lambda^{\prime},
		\end{align*}
		where we set $\lambda^{\prime}=\frac{\Theta(x_{0})}{2\pi}$. 
		Therefore, when $x_{0}=p_{1}$, we obtain
		$$
		( \lambda^{\prime}  )^{2} -2\lambda^{\prime}= (\beta_1)^2+2\beta_1. 
		$$
		which is equivalent to 
		\begin{align*}
		( \lambda^{\prime} +\beta_{1} )(   \lambda^{\prime} -\beta_{1}-2  )=0. 
		\end{align*}
		
		If $\lambda^{\prime}=2s-\beta_{1}$ (when \textbf{Case 1.1} holds), then
		\begin{align*}
		2s(    2s-  2 \beta_{1}-2        )=0,
		\end{align*}
		which yields that $s=\beta_{1}+1$, and
		\begin{align*}
		\Theta(x_{0})=2\pi \beta_{1}+4 \pi. 
		\end{align*}
		
		If $\lambda^{\prime}=2s+\beta_{1}$ (when \textbf{Case 1.2} holds), then
		\begin{align*}
		(2s+  2 \beta_{1} )(2s-2)=0,
		\end{align*}
		which yields that $s=1$, and 
		\begin{align*}
		\Theta(x_{0})=2\pi \beta_{1}+4 \pi. 
		\end{align*}
		
		When $x_{0} \notin \{p_{1},\cdots,p_{m}\}$ (when \textbf{Case 2} holds), $\lambda^{\prime}=2s$, 
		then we obtain
		\begin{align*}
		( \lambda^{\prime}  )^{2} -2\lambda^{\prime}=0,
		\end{align*}
		which yields that $s=1$, and 
		\begin{align*}
		\Theta(x_{0})=4 \pi. 
		\end{align*}
		
		\textbf{Step 2:} 
		We prove (c). By (b), we know that $\{u_{k}\}$ only has bubbles at 1-level, then by (a) and Proposition \ref{pro uk not converge and one level}, there exists a set $I \subset \{1,\cdots,m\}$ such that
		\begin{align*}
		s =	\frac{1}{2} \chi(\mathbb{S}^{2},\beta)-\sum_{i \in I} \beta_{i}, 
		\end{align*}
		which yields an upper bound of $s$ immediately:
		\begin{align*}
		s=1+\frac{1}{ 2} \sum_{ i \in \{1,\cdots,m\} \setminus I }\beta_{i}-\frac{1}{2} \sum_{i \in I} \beta_{i} \leq 1+\frac{1}{2} \sum_{i=1}^{m} |\beta_{i}|. 
		\end{align*}

	\end{proof}

	\subsection{Nonexistence of solutions around certain prescribed date} 	

	We give the proof of Theorem \ref{thm openness of nonexistence K C0}. 
	\begin{proof}
		We first show the assertion holds under the assumption ($\mathscr{O}_1$). 
		
		Assuming the contrary, there exist $K_{k} \in C^{+}(\mathbb{S}^{2}) $, $\beta_{k}=( \beta_{1}^{k}, \cdots , \beta_{m}^{k}      ) \in (-1,\infty)^{m}$ such that  $K_{k} \rightarrow K$ in $C(\mathbb{S}^{2} )$, $\beta_{k} \rightarrow \beta$, and $u_{k}$ solves 
		\begin{align*}
		-\Delta_{\mathbb{S}^{2}} u_{k}=K_{k}e^{2u}-2\pi\sum_{i=1}^{m} \beta_{i}^{k} \delta_{p_i}-1.
		\end{align*}
		By ($\mathscr{O}_1$), 
		$\sum_{i=1}^{m} \beta_{i} \in (-2,0)$.
		If $\{u_{k} \}$ has at least one bubble, then by Corollary \ref{cor equation of v if chi<2}, there exist $s<m$ and 
		$1 \leq i_{1}<i_{2}<\cdots<i_{s} \leq m$ such that for any $1 \leq j \leq s$, $\beta_{i_{j}}<0$ and 
		\begin{align*}
		4 \pi s+ 4 \pi \sum_{j=1}^{s} \beta_{i_{j}}= 2\pi\sum_{i=1}^{m} \beta_i+4 \pi .
		\end{align*}
		which contradicts to the assumption: $\beta \in 	\mathcal{A}_m$. 
		
		In conclusion, $\{u_{k}\}$ has no bubble. 
		Applying Proposition  \ref{prop convergence.with singular point}, $u$ solves 
		\begin{align*}
		-\Delta_{\mathbb{S}^{2}} u=Ke^{2u}-2\pi\sum_{i=1}^{m} \beta_i\delta_{p_i}-1,
		\end{align*}
		a contradiction to our assumptions. 
		
		By applying similar arguments and Corollary \ref{cor equation of v if betai leq 1}, the assertion still holds under assumption ($\mathscr{O}_2$). 
	\end{proof}
	
	Applying similar arguments as in the proof of Theorem \ref{thm openness of nonexistence K C0} and Proposition \ref{pro v's equation when Kk converges in C1}, we can prove Theorem \ref{thm openness of nonexistence K C1}.

	The result below follows from Theorem \ref{thm openness of nonexistence K C1} and \cite[Theorem 5]{Dey2018}. 
	\begin{cor}
		\label{cor dey}
		Assume $\beta=( \beta_{1}, \cdots, \beta_{m} ) \in \mathcal{A}_{m}$ satisfying
		\begin{align*}
		d_{1}(\beta,\mathbb{Z}_{o}^{m})=1, \quad \beta_{i} -1 \notin \mathbb{Z}.
		\end{align*}
		Then for any distinct $p_{1}, \cdots, p_{m}$, there is a neighbourhood $U$ of $\beta$ in $(-1,\infty)^{m}$, such that for any $\tilde{\beta} \in U$, the divisor $\sum_{i=1}^{m} \tilde{\beta}_{i} p_{i}$ cannot be represented by any metric in $\mathcal{M}(\mathbb{S}^{2})$ with constant curvature $1$. 
	\end{cor}
	
	\vspace{.1cm}
	
	\noindent{\bf Example.} 
	Let 
	$$
	\beta=( \beta_{1},\beta_{2},\beta_{3},\beta_{4} )=\left(-\frac{3-2 \alpha}{10}, 2k+\frac{\alpha-1}{10},\frac{\alpha-1}{10},2k+\frac{1}{10}\right),
	$$
	where $k \in \mathbb{N} \cup \{0\}$ and $\alpha \in (0,\frac{1}{100})$. 
	Then
	\begin{align*}
	&d_1(\beta,\mathbb{Z}_0^4)=d_1(\beta,(-1,2k,0,2k))=1, \\
	& \frac{1}{2} \chi(\mathbb{S}^{2},\beta)=1+2k+\frac{\alpha-1}{5} \in (0,+\infty) \setminus \mathbb{N}, \\
	& \big[ \frac{1}{2} \chi(\mathbb{S}^{2},\beta)\big]=2k, \\
	& \mathbb{F}(\beta)=\{ i: \beta_{i} \leq 2k+\frac{\alpha-1}{5}   \}. 
	\end{align*}
	
	\textbf{Case 1:} when $k=0$, $\mathbb{F}(\beta)=\{1\}$, and 
	\begin{align*}
	\frac{1}{2} \chi(\mathbb{S}^{2},\beta)-\beta_{1}=\frac{11}{10} \neq 1. 
	\end{align*}
	
	\textbf{Case 2:} when $k \geq 1$, $\mathbb{F}(\beta)=\{1,3\}$, and 
	\begin{align*}
	&\frac{1}{2} \chi(\mathbb{S}^{2},\beta)-\beta_{1}=2k+\frac{11}{10} \neq 1,\cdots,1+2k, \\
	&\frac{1}{2} \chi(\mathbb{S}^{2},\beta)-\beta_{3}=1+2k+\frac{ \alpha-1}{10} \neq 1, \cdots,1+2k, \\
	&\frac{1}{2} \chi(\mathbb{S}^{2},\beta)-\beta_{1}-\beta_{3}=1+2k-\frac{\alpha}{10} \neq 2,\cdots,2+2k. 
	\end{align*}
	
	In conculsion, $\beta \in 	\mathcal{A}_4\cap \{d_1(\beta,\mathbb{Z}_o^4)=1\}$. In other words, 
	$$
	\mathcal{A}_4\cap \{d_1(\beta,\mathbb{Z}_o^4)=1\}\neq\emptyset.
	$$
	All conditions in Corollary \ref{cor dey} are fulfilled, therefore there is a neighbourhood $U$ of $\beta$ such that no divisor in $U$ can be represented by a spherical metric.

	\bibliographystyle{plain}
	\bibliography{CM2}

\begin{thebibliography}{10}

\bibitem{bartolucci2011supercritical}
Daniele Bartolucci, Francesca De~Marchis, and Andrea Malchiodi.
\newblock Supercritical conformal metrics on surfaces with conical
  singularities.
\newblock {\em International Mathematics Research Notices},
  2011(24):5625--5643, 2011.

\bibitem{BCJM}
Daniele Bartolucci, Changfeng Gui, Aleks Jevnikar, and Amir Moradifam.
\newblock A singular sphere covering inequality: uniqueness and symmetry of
  solutions to singular liouville-type equations.
\newblock {\em Mathematische Annalen}, 374:1883–1922, 2019.

\bibitem{Bartolucci2007}
Daniele Bartolucci and Eugenio Montefusco.
\newblock Blow-up analysis, existence and qualitative properties of solutions
  for the two-dimensional {E}mden--{F}owler equation with singular potential.
\newblock {\em Mathematical methods in the applied sciences},
  30(18):2309--2327, 2007.

\bibitem{bartolucci2002}
Daniele Bartolucci and Gabriella Tarantello.
\newblock Liouville type equations with singular data and their applications to
  periodic multivortices for the electroweak theory.
\newblock {\em Communications in mathematical physics}, 229:3--47, 2002.

\bibitem{Brezis1991}
Ha{\"\i}m Brezis and Frank Merle.
\newblock Uniform estimates and blow--up behavior for solutions of $ -{ \Delta}
  u= {V} (x) e^{u}$ in two dimensions.
\newblock {\em Communications in partial differential equations},
  16(8-9):1223--1253, 1991.

\bibitem{ChenLin03}
Chiun-Chuan Chen and Chang-Shou Lin.
\newblock Topological degree for a mean field equation on riemann surfaces.
\newblock {\em Commun.Pure Appl. Math.}, 56(12):1667–1727, 2003.

\bibitem{ChenLin15}
Chiun-Chuan Chen and Chang-Shou Lin.
\newblock Mean field equation of liouville type with singular data: topological
  degree.
\newblock {\em Comm. Pure Appl. Math.}, 68(6):887–947, 2015.

\bibitem{C-L1}
Jingyi Chen and Yuxiang Li.
\newblock Homotopy classes of harmonic maps of the stratified 2-spheres and
  applications to geometric flows.
\newblock {\em Advances in Mathematics}, 263:357--388, 2014.

\bibitem{C-L2}
Jingyi Chen and Yuxiang Li.
\newblock Uniform convergence of metrics on {Alexandrov} surfaces with bounded
  integral curvature.
\newblock 2022.

\bibitem{CLWnegative}
Jingyi Chen, Yuxiang Li, and Yunqing Wu.
\newblock Prescribing negative curvature with cusps and conical singularities
  on compact surface.
\newblock 2024.

\bibitem{CW-L}
Wenxiong Chen and Congming Li.
\newblock Classification of solutions of some nonlinear elliptic equations.
\newblock {\em Duke Mathematical Journal}, 63(3):615 -- 622, 1991.

\bibitem{Dey2018}
Subhadip Dey.
\newblock Spherical metrics with conical singularities on 2-spheres.
\newblock {\em Geometriae Dedicata}, 196:53--61, 2018.

\bibitem{ding1999existence}
Weiyue Ding, J{\"u}rgen Jost, Jiayu Li, and Guofang Wang.
\newblock Existence results for mean field equations.
\newblock In {\em Annales de l'Institut Henri Poincar{\'e} C, Analyse non
  lin{\'e}aire}, volume~16, pages 653--666. Elsevier, 1999.

\bibitem{Eremenko2004}
Alexandre Eremenko.
\newblock Metrics of positive curvature with conic singularities on the sphere.
\newblock {\em Proceedings of the American Mathematical Society},
  132(11):3349--3355, 2004.

\bibitem{Eremenko2017}
Alexandre Eremenko.
\newblock Co-axial monodromy.
\newblock {\em Annali della Scuola Normale Superiore di Pisa. Classe di
  Scienze. Serie V}, (2):619--634, 2020.

\bibitem{gilbarg1977elliptic}
David Gilbarg and Neil~S Trudinger.
\newblock {\em Elliptic partial differential equations of second order}, volume
  224.
\newblock Springer, 1977.

\bibitem{GM}
Changfeng Gui and Amir Moradifam.
\newblock The sphere covering inequality and its applications.
\newblock {\em Invent. math.}, 214:1169–1204, 2018.

\bibitem{kapovich}
Michael Kapovich.
\newblock Branched covers between spheres and polygonal inequalities in
  simplicial trees.
\newblock {\em preprint}, 273, 2017.

\bibitem{li1999harnack}
Yanyan Li.
\newblock Harnack type inequality: the method of moving planes.
\newblock {\em Communications in Mathematical Physics}, 200:421--444, 1999.

\bibitem{li-Shafrir1994}
Yanyan Li and Itai Shafrir.
\newblock Blow-up analysis for solutions of $-{ \Delta} u= {V} (x) e^{u} $ in
  dimension two.
\newblock {\em Indiana University Mathematics Journal}, 43(4):1255--1270, 1994.

\bibitem{Lin}
C.S. Lin.
\newblock Topological degree for mean field equations on ${S}^2$.
\newblock {\em Duke Math. J.}, 104(3):501–536, 2000.

\bibitem{luo1992liouville}
Feng Luo and Gang Tian.
\newblock Liouville equation and spherical convex polytopes.
\newblock {\em Proceedings of the American Mathematical Society},
  116(4):1119--1129, 1992.

\bibitem{Mazzeo-Zhu}
Rafe Mazzeo and Xuwen Zhu.
\newblock Conical metrics on {R}iemann surfaces, {I}: the compactified
  configuration space and regularity.
\newblock {\em Geometry \& Topology}, 24(1):309--372, 2020.

\bibitem{Mondello-P}
Gabriele Mondello and Dmitri Panov.
\newblock Spherical metrics with conical singularities on a 2-sphere: angle
  constraints.
\newblock {\em International Mathematics Research Notices},
  2016(16):4937--4995, 2016.

\bibitem{SSTW}
Yuguang Shi, Jiacheng Su, Gang Tian, and Dongyi Wei.
\newblock Uniqueness of the mean field equation and rigidity of hawking mass.
\newblock {\em Calc. Var.}, 58(41), 2019.

\bibitem{Tarantello2005}
Gabriella Tarantello.
\newblock A quantization property for blow up solutions of singular
  {L}iouville-type equations.
\newblock {\em Journal of Functional Analysis}, 219(2):368--399, 2005.

\bibitem{Troyanov1988}
Marc Troyanov.
\newblock Metrics of constant curvature on a sphere with two conical
  singularities.
\newblock {\em Differential Geometry: Proceedings of the 3 rd International
  Symposium}, pages 296--306, 1988.

\bibitem{Troyanov1991}
Marc Troyanov.
\newblock Prescribing curvature on compact surfaces with conical singularities.
\newblock {\em Transactions of the American Mathematical Society},
  324(2):793--821, 1991.

\end{thebibliography}

\end{document}